\documentclass[reqno,11pt]{elsarticle}
\usepackage{amsfonts}
\usepackage{amsthm,amsmath,eucal,bbm}
\usepackage{enumerate}
\usepackage{fancyhdr}
\usepackage{amssymb}
\usepackage{stmaryrd}
\usepackage{fullpage}
\usepackage{appendix}
\usepackage{graphics}
 \newtheorem{theorem}{Theorem}[section]
 
 \newtheorem{lemma}[theorem]{Lemma}
 \newtheorem{proposition}[theorem]{Proposition}

 \theoremstyle{definition}
 \newtheorem{definition}[theorem]{Definition}
 \theoremstyle{definition}

 \newtheorem{remark}[theorem]{Remark}
 \theoremstyle{remark}

 \newcommand{\bel}{\begin{equation}\label}
 \newcommand{\enl}{\end{equation}}
 \newcommand{\beq}{\begin{eqnarray*}}
 \newcommand{\enq}{\end{eqnarray*}}

 \newcommand{\p}{\partial}

 \newcommand{\R}{\mathbb{R}}
 
 \newcommand{\N}{\mathbb{N}}
 
 \newcommand{\norm}[1]{\Vert#1\Vert}
 
 \newcommand{\abs}[1]{\left\vert#1\right\vert}
 \newcommand{\set}[1]{\left\{#1\right\}}
 
 \newcommand{\inner}[1]{\left(#1\right)}

 \newcommand{\com}[1]{\big[#1\big]}
 \newcommand{\reff}[1]{(\ref{#1})}

\begin{document}

\begin{frontmatter}

\title
{Regularity  of  traveling periodic stratified  water waves with vorticity}

\author{Ling-Jun Wang \corref{cor1}}

\ead{wanglingjun@wust.edu.cn}

 \address{School of Science, Wuhan University of Science and
   Technology, 430065 Wuhan, China}

\cortext[cor1]{Corresponding author. Tel:  +86 18062409592.   Fax: +86 2768893261.
}

\begin{abstract}
We prove real analyticity of all the streamlines, including the free surface, of a  steady stratified flow of water over a flat bed in the absence of stagnation points, with a H\"{o}lder continuous Bernoulli
function and a H\"{o}lder continuously differentiable  density
function. Furthermore,   we show that if the Bernoulli
function and the density function possess some Gevrey regularity of
index s, then the stream function admits the same Gevrey regularity throughout the fluid domain;
in particular if the Gevrey index s equals to 1, then we obtain analyticity of the stream function.
The regularity results hold for three distinct physical regimes: capillary, capillary-gravity, and gravity water waves.
\end{abstract}

\begin{keyword}
  Analyticity, Gevrey regularity, wave profile, stratified water wave, vorticity
\MSC[2010] 35B65, 35Q35, 76B03, 76B15
\end{keyword}

\end{frontmatter}

\section{Introduction}

We study  regularity of the streamlines, including the surface
profile, and regularity of the stream function, of stratified flows. Stratified flows are heterogeneous flows where the density varies as a function of
the streamlines \cite{Long53,Turner81,Yih60}. Stratification is a physically significant phenomenon for
certain
flows, where the fluid density may be
caused to fluctuate
 by numerous factors such as the interplay between
gravity and the salinity of the water. Recently, in \cite{Walsh1, Walsh2}, the author developed an existence theory for two-dimensional stratified steady and
periodic gravity waves, with or without surface tension. Using local and global bifurcation techniques, it was shown that, for stratified flows without stagnation points, there exist  both small and large amplitude traveling periodic waves. Even more recently, some results ensuring the
local existence of stratified waves, allowing for both surface tension and stagnation points,
were proven by the authors in \cite{HM}, extending the work of \cite{EMM}. However, in the analysis of
this paper, we will invoke the assumption of there being no stagnation points.

For homogeneous water waves, that is, waves with constant density, a series of works analyze  the a priori regularity of the streamlines, and
particularly the free-surface. For instance in the irrotational setting Lewy \cite{MR0049399} showed that irrotational waves
without stagnation points have real analytic profiles. Recent developments, proving the regularity of streamlines
for rotational flows, were initiated by Constantin and Escher in \cite{MR2753609} in the setting of homogeneous gravity waves over a
flat bed. Under the assumption that the vorticity function is  H\"{o}lder continuously differentiable, it was shown in \cite{MR2753609} that, each streamline, except the free surface, is real analytic; if further the vorticity
function is real analytic, then the free surface itself is also analytic. The arguments in \cite{MR2753609} base on
translational invariance property of the resulting elliptic operator in the direction of wave propagation, and
the celebrated result due to Kinderlehrer et al. \cite{MR531272} on regularity for elliptic free boundary problems. Later on, similar  strong a priori regularity was established for a wide variety of homogeneous regimes, see for instance \cite{MR2763714,HenryJmfm,HenryPro} for periodic gravity waves with surface tension, \cite{MaQam} for deep-water waves, \cite{MaImrn} for flows with merely bounded
vorticity, \cite{HurImrn} for solitary-water waves, and the survey article \cite{Esurvey}. In all the aforementioned works the analyticity of free surface is established under the extra assumption that the vorticity function is analytic. Quite recently, a similar result was established in \cite{LW} without the analyticity assumption on the vorticity function. Precisely, the authors in \cite{LW} showed that if the vorticity function is only
H\"{o}der continuous, all the streamlines, including the free surface,
of the steady homogeneous flow over a flat bed in the absence of stagnation points, are real analytic. The conclusions in \cite{LW} were achieved by using some a priori Schauder estimates and giving successively a quantitative bound for
each derivative of the streamlines in the H\"{o}der norm. Moreover, studied also in \cite{LW} was the case when the vorticity possesses more regularity property rather than
H\"{o}der continuity, namely Gevrey regularity of index $s$. Gevrey class is an intermediate space between
the spaces of smooth functions and analytic functions, and the Gevrey class function of index
1 is just the real-analytic function; see Subsection 2.2 below for precise definition of Gevrey class. It was shown in \cite{LW} that if the
vorticity is Gevrey regular, the stream function admits the same Gevrey regularity
in the fluid domain, up to the free surface.

For heterogeneous, or stratified, water waves, some regularity results of steady periodic
 waves without stagnation points, have been  studied recently in \cite{HMdcds}, in three distinct physical regimes, namely:
capillary, capillary-gravity, and gravity water waves. There the authors proved, for all three types of waves,
that, when the Bernoulli function is H\"{o}der
continuous and the variable density function
has a first derivative which is H\"{o}der
continuous, then the free-surface profile is the graph
of a smooth function. Furthermore, they showed that the streamlines are analytic a priori
for capillary stratified waves, whereas for gravity and capillary-gravity stratified waves the
streamlines are smooth in general, and analytic
in an unstable regime; moreover, if the
Bernoulli function and the streamline density function are both real analytic
functions then
all of the streamlines, including the wave profile, are real analytic
for all gravity, capillary, and capillary-gravity stratified waves.

In the present work we follow  the arguments in \cite{LW} for homogeneous water waves to study the regularity of the streamlines, including the surface
profile, for  stratified flows. We show, for the above three types of waves, that: 1) if the Bernoulli function is H\"{o}der
continuous and the density function
is H\"{o}der continuously differentiable, then all the streamlines,  including the wave profile, are real analytic; see Theorem \ref{thm1}; 2) if both the Bernoulli function and the  density function are in Gevrey (analytic) class, then the stream function admits the same Gevrey (analytic) regularity
in the fluid domain, up to the free surface; see Theorem \ref{thm2} stated in Subsection 2.2.

The paper is organized as follows. In Section 2 we formulate the rotational capillary-gravity stratified water-wave
problem as free boundary problem for stream function and its equivalent reformulation in a fixed
rectangular domain, and state our main regularity results. Notations and some useful inequalities
are listed. Section 3 is devoted to the proof of analyticity of streamlines including the free surface.
In Section 4 we study the Gevrey (analytic) regularity of stream function. In the last section, Section
5, we consider the traveling gravity water waves, and obtain similar regularity results
for streamlines and stream function.

\section{Preliminaries and the main results}

\subsection{The governing equations for stratified water waves}
Consider a steady two-dimensional flow of an incompressible inviscid fluid with variable density
and a steady wave on the free surface of the flow. By steady, we mean that the flow and the surface wave move at a constant speed
from left to right without changing their configuration; that is, the velocity field of flow and the surface wave exhibit an $(t,X)-$dependence
in the form of $X-ct$, where $X$ is the horizontal space variable, $c>0$ is the speed of the wave and $t$ denotes the time. Setting $x=X-ct$, we eliminate the time dependence of fluid flow and pass to the frame of reference moving with the wave.
Assume that the flow is over a flat bed $y=-d$ with $0<d<\infty$, the free surface is given by $y=\eta(x)$ which
is oscillating around the line $y = 0$, and  the liquid occupy the stationary domain
\[
  \Omega=\{(x,y)\in \R^2: -d<y<\eta(x)\}.
\]
Also, the flow is assumed to be driven by capillarity (that is, surface tension) on the surface and gravity acting on the body of the fluid.

Let
$u = u(x, y)$ and $v = v(x, y)$ denote the horizontal and vertical velocities, respectively, and let
$\tilde\rho =\tilde \rho(x, y) > 0$ be the density.
Define the \textit{(relative) pseudo-stream function} $\psi(x,y)$ by
\begin{equation}\label{def psi}
  \psi_y=\sqrt{\tilde\rho}(u-c),\quad \psi_x=-\sqrt{\tilde\rho}v.
\end{equation}
The level sets $\{(x,y):\psi(x,y)={\rm constant}\}$ are called \textit{streamlines} of the fluid motion. The above relation \reff{def psi} determines $\psi$ up to a constant. For definiteness we choose
  $\psi=0$ on the free boundary, so that $\psi=-p_0$ on $y=-d$, where $p_0<0$ is the \textit{(relative) pseudo-voumetric mass flux}:
  \begin{eqnarray*}
  p_0=\int_{-d}^{\eta(x)}\psi_y(x,y)~dy.
\end{eqnarray*}
  Since $\rho$ is transported, it must be constant on the streamlines and hence, we may think of it as a function
of $\psi$ and assume
\[
  \tilde\rho(x,y)=\rho(-\psi(x,y)),
\]
where $\rho: [p_0,0]\rightarrow \R^+$ is referred as   \textit{ streamline density function}. Finally, consider only the case where there are no stagnation points throughout the fluid domain, that is,
\begin{equation}\label{no stag}
  \psi_y(x,y)\leq -\delta<0\quad {\rm in}~~  \bar{\Omega}
\end{equation} for some $\delta>0$. Then the governing equations for the capillary-gravity stratified water wave problem are formulated as
\begin{subequations}\label{EquPsi}
\begin{eqnarray}
 \triangle\psi-gy\rho'(-\psi)=-\beta(-\psi),&\quad (x,y)\in \Omega_{\eta},\label{EquPsi1}\\
 \abs{\nabla \psi}^2+2g(y+d)\rho-2\sigma\frac{\eta_{xx}}{(1+\eta_x^2)^{\frac{3}{2}}}=Q,&\quad y=\eta(x),\label{EquPsi2}\\
  \psi=0,&\quad y=\eta(x),\label{EquPsi3}\\
  \psi=-p_0, &\quad y=-d.\label{EquPsi4}
 \end{eqnarray}
\end{subequations}
Here $\beta: [p_0,0]\rightarrow \R$ is the \textit{Bernoulli function}, $g>0$ is the gravitational constant of acceleration, $\sigma>0$ is the coefficient of surface tension, and $Q$ is  a constant related to the energy.
We refer to \cite{Walsh2} for the detailed derivation of the above system of governing equations.

The equation \reff{EquPsi1} is known as Yih's equation or the Yih-Long equation. The wave profile $\eta(x)$ represents an unknown in the problem since it is a free surface.
Note that \reff{EquPsi3} and \reff{EquPsi4} imply that the free surface and flat bed are each level sets of  $\psi$, and thus are
streamlines. The system \reff{EquPsi} with $g=0$ corresponds to the capillary stratified water waves.

\subsection{Statement of the main results}
To state our main results, we first recall the definition of Gevrey
class functions, which is an intermediate space between the spaces of
smooth functions and real-analytic functions;  see \cite{MR1249275}
for more detail.

\begin{definition}\label{def gevrey}
 Let $W$ be an open subset of $\mathbb{R}^d$ and $f$
be a real-valued  function defined on the closure $\bar W$ of $W$.  We say $f$ belongs to  Gevrey class
in $\bar {W}$ of index $s\geq 1$,  denoted by $f\in G^s(\bar W)$,   if $f\in
C^\infty(\bar W)$  and for any $x_0\in\bar W$  there exists a neighborhood
$U$ of $x_0$ such that for any $x\in U\cap\bar W$, the series
\[
   \sum_{\alpha\in\mathbb N^d}\frac{\partial^\alpha f(x_0)}{(\alpha!)^s}(x-x_0)^\alpha
\]
converges to $f(x)$.
\end{definition}

Note that $f$ is real analytic in
$\bar W$ if
$s=1$.  Observe  $ \abs\alpha!\leq 2^{\abs\alpha} \alpha!$ for any multi-index $\alpha\in\N^d$. Then as an alternative characterization for Gevrey class function,
we have $f\in G^s(\bar W)$ if    for any compact subset $K\subset\bar W$, there
exists a constant $C_K$, depending only on $K$,  such that
\begin{eqnarray*}
\forall~\alpha\in\mathbb \N^d,\quad \max_{x\in K}|\partial^\alpha{f}(x)| \leq C_K^{|\alpha|+1}(|\alpha|!)^s.
\end{eqnarray*}
In this paper we will derive a stronger estimate
than the above one,
namely,
\[
\forall~\alpha\in\mathbb N^d,\quad \max_{x\in
  \bar W}|\partial^\alpha{f}(x)| \leq
C^{|\alpha|+1}(|\alpha|!)^s.
\]

Throughout the paper let  $C^{k,\mu}(\bar W)$,
$k\in\mathbb N, \mu\in(0,1)$,   be
the  standard H\"{o}lder  space of functions $f : \bar W \rightarrow \mathbb R$ with H\"{o}lder-continuous derivatives
of exponent $\mu$ up to order $k$.      For given $p_0<0$ ,
$\beta\in C^{3,\mu}([p_0,0])$ and $\rho\in C^{1,\mu}([p_0,0])$,  the existence of periodic water waves with or without surface tension has been
established in \cite{Walsh2,Walsh1}.
 Our main result below shows that, with a H\"{o}lder continuous
Bernoulli function and a H\"{o}lder continuously differentiable density function,
each streamline can be described by the graph of some analytic function.

\begin{theorem}\label{thm1}
Let $\psi(x,y)$ be the pseudo-stream function for the boundary problem
\reff{EquPsi1}-\reff{EquPsi4} with free surface $y=\eta(x)$.  Suppose  $\beta\in
C^{0,\mu}([p_0, 0])$ and $\rho \in C^{1,\mu}([p_0, 0])$ with $p_0<0$ and $0<\mu<1$  given. Then each
streamline including the free surface $y=\eta(x)$ is a real-analytic curve.
\end{theorem}

The following result shows that the pseudo-stream function admits the same
regularity as the vorticity.

\begin{theorem}\label{thm2}
Let $\psi(x,y)$ be the pseudo-stream function for the boundary problem
\reff{EquPsi1}-\reff{EquPsi4} with free surface $y=\eta(x)$.  Suppose  $\beta,\rho \in
G^s([p_0, 0])$ with $p_0<0$ and $s\geq1$ given. Then  $\psi(x,y) \in
G^s(\bar \Omega )$;  in particular if $s=1$, i.e., $\beta$ and $\rho$ are analytic in $[p_0,
0]$, then the pseudo-stream function $\psi(x,y)$ is analytic in $\bar\Omega$.
\end{theorem}

\begin{remark}
The above results  also hold for the traveling gravity
water waves; see Theorem \ref{th3} in Section \ref{sec5}. Moreover,
  although we focus on the capillary-gravity water waves in this  and the next two subsequent sections, our arguments are also applicable for  capillary waves, that is, waves driven only by the surface tension by setting $g=0$. Accordingly, the conclusions in these two sections also hold for the capillary waves.
\end{remark}

\subsection{Reformulation}

Under the no-stagnation assumption \reff{no stag}, we  can use the
partial hodograph change of variables  to transform  the free boundary problem \reff{EquPsi1}-\reff{EquPsi4} into a problem with fixed boundary.
Precisely,  if we  introduce the new variable $(q,p)$ with
  \[
   q=x, \quad p=-\psi(x,y),
\]
and exchange  the roles of the $y$-coordinate and
$\psi$ by setting
\begin{eqnarray*}
  h(q,p)=y+d,
\end{eqnarray*}
then the fluid domain $\Omega$ is transformed into a fixed infinite strip
\begin{eqnarray*}
   R=\{(q,p): q\in\R,\ p_0<p<0\},
\end{eqnarray*}
and the system \reff{EquPsi1}-\reff{EquPsi4} can be reformulated  in this strip as
\begin{subequations}
\begin{eqnarray}
 (1+h_q^2)h_{pp}-2h_ph_q h_{pq}+h_p^2h_{qq}-g(h-d)\rho_p
  h_p^3+\beta(p)h_p^3=0,&\quad {\rm in~~} R,\label{Equh1}\\
  1+h_q^2+\inner{2g\rho h-2\sigma\frac{h_{qq}}{(1+h^2_q)^{3/2}}-Q}h_p^2=0,&\quad {\rm on~~} p=0,\label{Equh2}\\
  h=0,&\quad {\rm on~~} p=p_0.\label{Equh3}
\end{eqnarray}
\end{subequations}
We refer to \cite{Walsh1,Walsh2}  for the equivalence of the two  systems
\reff{EquPsi1}-\reff{EquPsi4} and \reff{Equh1}-\reff{Equh3} of governing equations.
Note that $h_p=\frac{1}{c-u}$. The no-stagnation assumption \reff{no stag} ensures that
\begin{eqnarray}\label{hp bounded}
 0<\inf_{(q,p)\in\bar R}h_p\leq h_p\leq \sup_{(q,p)\in\bar R}h_p\leq\frac{1}{\delta}.
\end{eqnarray}

 The following proposition shows that the regularity is preserved
through hodograph transformation.  So we only need to study  the
above  problem \reff{Equh1}-\reff{Equh3} instead of
the original one \reff{EquPsi1}-\reff{EquPsi4}.

\begin{proposition}\label{reserv}
      Let $h\in C^{2,\mu}(\bar R)$ be  a solution to the
      problem \reff{Equh1}-\reff{Equh3} .  If the mapping  $q\mapsto
h(q,p)$,  with any fixed $p\in[p_0,0]$,  is  analytic in  $\mathbb R$,
then each  streamline including the free surface is an analytic
curve.   Moreover if  $h\in G^s(\bar R)$ then  the
      pseudo-stream function $\psi$ for \reff{EquPsi1}-\reff{EquPsi4}  lies in
      $G^s(\bar\Omega)$; in particular $\psi$ is analytic in
      $\bar\Omega$ provided $h$ is analytic in $\bar R$.
\end{proposition}

We refer to \cite{LW} for the proof of the above proposition, which is an application of \cite[Theorem 3.1]{clx2011}.
Some of the technical
results in the present paper are similar to those in \cite{LW},
to which we refer for their proofs.

\subsection{Notations and some useful inequalities}
We list some notations and useful inequalities which will be used
throughout the paper.   Let $k\in\N$ and $\mu\in(0,1)$,  and
let
$\inner{C^{k,\mu}(\bar R);
  \norm{\cdot}_{k,\mu;\bar R}}$ be  the standard H\"{o}lder space equipped
with the norm
\begin{eqnarray*}
  \norm{w}_{k,\mu;\bar R}=\sum_{\abs{\alpha}=0}^{k}\sup_{\bar R}\abs{\p^{\alpha}w(q,p)}+\sup_{\abs{\alpha}=k}\sup_{\stackrel{(q,p)\neq(\tilde q,\tilde p)}{ \bar R}}\frac{\abs{\p^{\alpha}w(q,p)-\p^\alpha w(\tilde q,\tilde p)}}{\abs{(q,p)-(\tilde q,\tilde p)}^\mu}.
\end{eqnarray*}
To simplify the notation we will use the notation
$\norm{\cdot}_{k,\mu}$ instead of $\norm{\cdot}_{k,\mu;\bar R}$
if no confusion occurs.  For the case when $\mu=0$, we naturally  define
\[
  \norm{w}_{k}=\norm{w}_{k,0}=\sum_{\abs{\alpha}=0}^{k}\sup_{\bar R}\abs{\p^{\alpha}w}.
\]
For $\mu\in(0,1)$, direct verification shows that
\begin{eqnarray}\label{algebra}
  \norm{uw}_{0,\mu}\leq \norm{u}_{0,\mu}\norm{w}_{0,\mu},\quad \norm{uw}_{1,\mu}\leq 2\norm{u}_{1,\mu}\norm{w}_{1,\mu}.
\end{eqnarray}

For a multi-index $\alpha=(\alpha_1,\alpha_2)\in\N^2$, we denote
$\p^\alpha=\p_q^{\alpha_1}\p_p^{\alpha_2}$,~  $\alpha!=\alpha_1!\alpha_2!$ and denote the length of $\alpha$ by $\abs{\alpha}=\alpha_1+\alpha_2$. Moreover for two multi-indices $\alpha$ and $\beta=(\beta_1,\beta_2)\in\N^2$, by
$\beta\leq \alpha $ we mean $\beta_i\leq \alpha_i$ for each $1\leq
i\leq 2$.  Let ${\alpha\choose\beta}$ be the binomial coefficient, i.e.,
\[
  {\alpha\choose\beta}=\frac{\alpha!}{\beta!(\alpha-\beta)!}=\frac{\alpha_1!\alpha_2!}{\beta_1!(\alpha_1-\beta_1)!\beta_2!(\alpha_2-\beta_2)!}.
\]
In the sequel, we use the convention that  $m!=1$ if $m\leq 0$.

We now list without proof some straightforward inequalities  to be
used frequently later.
\begin{enumerate}[(i)]
\item For any    $\beta\leq\alpha$, we have
\begin{equation}\label{}
{\alpha\choose\beta}=\frac{\alpha !}{\beta!(\alpha-\beta)!}\leq
\frac{\abs\alpha!}{\abs\beta!(\abs\alpha-\abs\beta)!}={\abs\alpha
\choose \abs\beta}.
\end{equation}
\item Given $s\geq 1$, then $\inner{m!}^{s-1}\inner{n!}^{s-1}\leq
  \com{(m+n)!}^{s-1}$.
\item Given $m\geq 1$, we have,  for any integer $k$ with $2\leq k\leq 3$,
  $$\sum_{0<j<m}\frac{\abs{m}^k}{\abs{j}^k(m-j)^k}\leq 2^k
  \pi^2.
  $$
\end{enumerate}

\begin{lemma}\label{lemSum}
Given $\alpha\in\N^2$, we have
\begin{eqnarray*}
\sum_{\beta\leq \alpha, 0<\abs
  \beta<\abs\alpha}\frac{\abs\alpha^2}{\abs\beta^3(\abs\alpha-\abs\beta)^2}&\leq&
8\pi^2,
\\
\sum_{\beta\leq \alpha, 0<\abs
  \beta<\abs\alpha}\frac{\abs\alpha^2}{\abs\beta^2(\abs\alpha-\abs\beta)^2}&\leq&
8\pi^2 \abs\alpha,
\\
\sum_{\beta\leq \alpha, 0<\abs
  \beta<\abs\alpha}\frac{\abs\alpha}{\abs\beta^3(\abs\alpha-\abs\beta)}&\leq&  8\pi^2,
  \\
  \sum_{\beta\leq \alpha, 0<\abs
  \beta<\abs\alpha}\frac{\abs\alpha^3}{\abs\beta^4(\abs\alpha-\abs\beta)^3}&\leq&
8\pi^2.
\end{eqnarray*}
\end{lemma}

We refer to \cite[Lemma 2.6]{LW} for the proof of the lemma.

\section{Analyticity of streamlines}\label{sec3}
We prove now the analyticity of streamlines, including the free
surface $y=\eta(x)$.  In view of  Proposition \ref{reserv},  it suffices to show the following
proposition, which concludes that the  map $q\mapsto h(q,p)$ is
analytic for all $p\in[p_0,0]$.

\begin{proposition}\label{propassu}
   Let $\beta\in C^{0,\mu}\inner{[p_0, 0]}$ and $\rho\in C^{1,\mu}\inner{[p_0, 0]}$ with $p_0<0$ and $0<\mu<1$
   given,  and $h\in C^{2,\mu}(\bar
   R)$ be a solution of the governing equations
   \reff{Equh1}-\reff{Equh3}.
Then there exists a constant $L\geq 1$,  such that for all $m\in\N$ with $m\geq 2$, one has
  \begin{equation}\label{Em}
   (E_m):\quad\quad \norm{\p^m_q h}_{2,\mu}\leq L^{m-1}(m-2)!.
  \end{equation}
Thus
 the map $q\mapsto h(q,p)$ is analytic for all $p\in[p_0,0]$.
\end{proposition}

\begin{remark}
 As to be seen in the proof below, the constant $L$ depends  on $\mu,\sigma,
\inf_{\bar R}h_p$,  $ \norm{h}_{2,\mu}$, $\norm{\beta}_{0,\mu}$, $\norm{\rho}_{1,\mu}$ and the number $\delta$ given in \reff{hp
bounded}, but
 independent of the order $m$ of derivative.
\end{remark}

\begin{remark}\label{remReg}
Starting from the $C^{2,\mu}$-regularity  solution $h$ of the governing equations
   \reff{Equh1}-\reff{Equh3},  we use the  Schauder estimate (
  cf. \cite[Theorem 6.30]{MR1814364}) for  $\p_q h$ which satisfies  a
  nonlinear elliptic equation of the same type as \reff{Equh1}-\reff{Equh3}, to conclude that  $\p_q h\in
  C^{2,\mu}(\bar R)$.   Repeating the procedure, we can derive by
  standard  iteration that $\p_q^k h\in
  C^{2,\mu}(\bar R)$  for any $k\in\N$;
  see for instance \cite{MR2027299, HMdcds}.
\end{remark}

To confirm the last statement in the above proposition \ref{propassu},
we   choose $C$ in such a way
that
\[
     C=\max\set{L,  \norm{h}_{1,\mu}},
\]
which,  along with the estimate $(E_m)$ with $m\geq 2$   in Proposition \ref{propassu},  yields
\begin{eqnarray*}
\forall~m\in\mathbb N, \quad \max_{(q,p)\in
  \bar R}|\partial_q^m{h}(q,p)| \leq C^{m+1}m! .
\end{eqnarray*}
In particular, for any  $p\in[p_0,0]$,  $\max_{q\in\mathbb R} \abs{\p_q^m h(q,p)}\leq C^{m+1}
m!$.  This gives the real analyticity of the map $q\mapsto h(q,p)$, $p\in[p_0,0]$.

Before proving the above proposition, we first give the following
technical  lemma and refer to \cite[Lemma 3.4]{LW} for its proof.

\begin{lemma}\label{stab}

Let $\ell=1$ or $2$ be given, and let $\norm{\cdot}$ stand for the
H\"{o}lder norm $\norm{\cdot}_{0,\mu}$ or  $\norm{\cdot}_{1,\mu}$. Suppose that $k_0$ is an integer
with $k_0\geq\ell+1$,  and $\p_q^k u_j\in
C^{0,\mu}(\bar R)$ for all $k\leq k_0$,  $j=1,2,3$.  If there
exists  a constant $H\geq 1$ such that
\begin{equation}\label{condition3}
 \forall~\ell+1\leq k\leq k_0,  \quad  \|\p_q^k u_j\|
  \leq H^{k-\ell}(k-\ell-1)!, \quad j=1,2,3,
\end{equation}
then we can find a constant $C_*$ depending only on $\ell$ such that
\begin{eqnarray*}
 \forall~\ell+1\leq k\leq k_0,  \quad
 \big\|\p_q^k\inner{u_1u_2u_3}\big\|\leq C_* \Big(\sum_{j=1}^3\norm{u_j}_{\ell+1,\mu}+1\Big)^3 H^{k-\ell}(k-\ell-1)!.
\end{eqnarray*}

\end{lemma}

We now prove Proposition \ref{propassu}.
\begin{proof}[Proof of Proposition \ref{propassu}]
  In view of Remark \ref{remReg} we may assume  that  $\p_q^k h\in
  C^{2,\mu}(\bar R)$ for any $k\in\N$.
Now we prove the validity of  $(E_m)$ by using induction on $m$.   For $m=2
  $, $(E_m)$ obviously holds if we choose
\begin{eqnarray*}
L\geq \norm{\p_q^2 h}_{2,\mu}+1.
\end{eqnarray*}
 Now assume that $(E_j)$ holds for all $j\in\N$ with $2\leq j\leq
  m-1$ and $m\geq 3$, that is,
\begin{equation}\label{nomind}
\norm{\p^j_q h}_{2,\mu}\leq L^{j-1}(j-2)!, \quad 2\leq
j\leq m-1.
\end{equation}
Then we   show the validity of $(E_m)$.   For this purpose,
taking  the derivative with respect to $q$ up to order  $m$   on
both sides of  equations \reff{Equh1}-\reff{Equh3},  and then applying
Leibniz formula, we have
\begin{eqnarray}\label{oprtAB}
\left\{
\begin{array}{lll}
  A(h)[\p_q^m h]=f_1+f_2 \quad & {\rm in}~~R,\\[3pt]
  B(h)[\p_q^m h]=\varphi_1+\varphi_2\quad  &{\rm on}~~p=0,\\[3pt]
  \p_q^m h=0 \quad & {\rm on}~~p=p_0,
  \end{array}
  \right.
  \end{eqnarray}
where the operators
\[
  A(h)[\phi]= (1+h_q^2)\phi_{pp}-2h_q h_p \phi_{qp}+h_p^2\phi_{qq}, \quad \quad B(h)[\phi]=2\sigma\frac{h_p^2}{(1+h_q^2)^{\frac{3}{2}}}\p_q^2
\]
and the right-hand side
\begin{eqnarray}
  f_1&=&\sum_{1\leq n\leq m} {m\choose n}\Big[-(\p_q^n h^2_q)(\p_q^{m-n}h_{pp})+2\inner{\p_q^n (h_ph_q)}(\p_q^{m-n} h_{pq})-(\p_q^n h^2_p)(\p_q^{m-n}h_{qq})\Big],\label{def f1}\\
  f_2&=&-\beta(p)(\p_q^m h^3_p)+g\rho_p\sum_{0\leq n\leq m}{m\choose n}(\p_q^n (h-d))(\p_q^{m-n}h_p^3),\label{def f2}\\
  \varphi_1&=&\p_q^m h_q^2+2g\rho \p_q^m(h h_p^2)-Q( \p_q^m h_p^2),\label{defvp1}\\
  \varphi_2&=&-2\sigma \sum_{1\leq n\leq m}(\p_q^{m-n} h_{qq})\inner{\p_q^n \frac{h_p^2}{(1+h_q^2)^{\frac{3}{2}}}}.\label{defvp2}
\end{eqnarray}
The operator $A(h)$ is uniformly elliptic since its coefficients satisfy
\begin{eqnarray*}
  (1+h_q^2)h_p^2-h^2_q h^2_p=h_p^2\geq \inf_{\bar R}h_p^2>0
\end{eqnarray*}
due to \reff{hp bounded}. Moreover it has been shown in \cite{HenryJmfm,HMdcds} that the operator $B(h)$ satisfies the complementing condition in the sense of \cite{MR0125307}. Since $h\in C^{2,\mu}(\bar R)$  the coefficients of the operators $A(h)$ and $B(h)$ are in $C^{1,\mu}(\bar R)$. Moreover, by virtue of the induction assumption \reff{nomind}, one has $\p_q^i\p_p^j h\in C^{0,\mu}(\bar R)$ for all multi-index $(i,j)$ with $i+j\leq m+1$ and $j\leq 2$, and similarly $\p_q^i\p_p^j h\in C^{1,\mu}(\bar R)$ for all multi-index $(i,j)$ with $i+j\leq m$ and $j\leq 1$. As a result,
the right-hand side $f_i,\varphi_i\in C^{0,\mu}(\bar R)$, $i=1,2$,   since by \reff{algebra} the
product of two functions in $C^{k,\mu}(\bar R)$ is still in
$C^{k,\mu}(\bar R)$ with $k=0,1$. Thus, the  standard Schauder
estimate (see for instance \cite[]{MR0125307})
\begin{eqnarray}\label{par m h}
  \norm{\p_q^m h}_{2,\mu}\leq \mathcal C \inner{\norm{\p_q^m h}_{0}+\sum_{i=1}^2\norm{f_i}_{0,\mu}+\sum_{i=1}^2\norm{\varphi_i}_{0,\mu}}
\end{eqnarray}
 holds, where $\mathcal C$ is a constant depending only on $\mu,
\delta,\inf_{\bar R}h_p$ and $ \norm{h}_{2,\mu}$.  To show
 $(E_m)$ is valid, we estimate the terms on the  right-hand
 side of \reff{par m h} through the following
 steps.

To simplify the notations,  we will use $C_j, j\geq 1$, to denote
suitable {\it  harmless constants} larger than 1. By harmless
constants we mean that they  are
independent of  $m$.

{\it Step 1)}~We claim that there exists $C_1>0$ such that, with $m\geq 3$,
\begin{equation}\label{step1}
\norm{\p_q^m h}_{0}\leq C_1 L^{m-2}(m-2)!.
\end{equation}
Indeed, when  $m=3$ the above estimate obviously holds if we choose
$C_1=\norm{h}_{3,\mu}+1$; when $m\geq 4$
 it follows  from the induction assumption \reff{nomind}  that
 \begin{eqnarray*}
   \norm{\p_q^m h}_{0}\leq \norm{\p_q^{m-2}
     h}_{2,\mu}\leq L^{m-3}(m-4)!\leq  L^{m-2}(m-2)!.
 \end{eqnarray*}
Then \reff{step1} follows.

{\it Step 2)}~Let $f_1$ be given in \reff{def f1}.  In this step we  prove
\begin{equation}\label{step2}
\norm{f_1}_{0,\mu}\leq  C_{2} L^{m-2}(m-2)!.
\end{equation}
Observe that ,   by \reff{algebra},
\begin{eqnarray}\label{est f1}
  \begin{split}
    \norm{f_1}_{0,\mu}& \leq& \sum_{1\leq n\leq m} {m\choose
       n}\norm{\p_q^n h^2_q}_{0,\mu}\norm{\p_q^{m-n}h_{pp}}_{0,\mu}+2\sum_{1\leq n\leq m} {m\choose
       n}\norm{\p_q^n (h_ph_q)}_{0,\mu} \norm{\p_q^{m-n} h_{pq}}_{0,\mu}\\
&&+\sum_{1\leq n\leq m} {m\choose
       n}\norm{\p_q^n h^2_p}_{0,\mu}  \norm{\p_q^{m-n}h_{qq}}_{0,\mu}.
   \end{split}
\end{eqnarray}
We now treat the first term on the right-hand side, and write
   \begin{eqnarray}\label{est f1 1}
\begin{split}
     &\sum_{1\leq n\leq m} {m\choose
       n}\norm{\p_q^n
       h^2_q}_{0,\mu}\norm{\p_q^{m-n}h_{pp}}_{0,\mu}\leq
\sum_{1\leq n\leq m} {m\choose
       n}\norm{\p_q^n
       h^2_q}_{0,\mu}\norm{\p_q^{m-n}h}_{2,\mu}\\
&\leq \inner{\sum_{1\leq
         n\leq 2} +\sum_{3\leq n\leq m-2 } +\sum_{ m-1\leq n\leq m}
     }{m\choose n}\norm{\p_q^n
       h^2_q}_{0,\mu}\norm{\p_q^{m-n}h}_{2,\mu}.
\end{split}
 \end{eqnarray}
 By the induction assumption \reff{nomind}, one has
 \begin{eqnarray*}
  \forall~3\leq n \leq m, \quad \norm{\p_q^n h_q}_{0,\mu}\leq \norm{\p_q^{n-1} h}_{2,\mu}\leq L^{n-2}(n-3)!.
 \end{eqnarray*}
 Thus applying Lemma \ref{stab}, with $\ell=2$, $k_0=m$, $H=L$,
$u_1=u_2=h_q$ and $u_3=1$, yields that
 \begin{eqnarray}\label{par n hq2}
   \forall ~3\leq  n \leq m,\quad \norm{\p_q^n h^2_q}_{0,\mu}\leq C_5 L^{n-2}(n-3)!.
 \end{eqnarray}
 Moreover, we have
 \begin{eqnarray}\label{q m-n h}
  \forall~1\leq n\leq m-2,\quad \norm{\p_q^{m-n}h}_{2,\mu}\leq L^{m-n-1}(m-n-2)!
 \end{eqnarray}
due to the induction assumption \reff{nomind}.
Then using the above two estimates, straightforward verification shows that
\begin{equation}\label{1}
\sum_{1\leq n\leq 2} {m\choose n}\norm{\p_q^n
  h^2_q}_{0,\mu}\norm{\p_q^{m-n}h}_{2,\mu}+\sum_{m-1\leq n\leq m} {m\choose n}\norm{\p_q^n
  h^2_q}_{0,\mu}\norm{\p_q^{m-n}h}_{2,\mu}
\leq C_{6}  L^{m-2}(m-3)!
\end{equation}
if we choose
\begin{eqnarray*}
 C_{6}\geq(\norm{h}_{3,\mu}+1)(30\norm{h}_{3,\mu}+4C_5+6).
\end{eqnarray*}
Next for the case when $3\leq n\leq m-2$, which  appears only when $m\geq 5$,
combination of the estimates \reff{par n hq2} and \reff{q m-n h} gives
\begin{eqnarray*}
\sum_{3\leq n\leq m- 2} {m\choose n}\norm{\p_q^n
  h^2_q}_{0,\mu}\norm{\p_q^{m-n}h}_{2,\mu}
&\leq & C_5 \sum_{3\leq n\leq m- 2} \frac{m!}{n!(m-n)!} L^{n-2}(n-3)!
L^{m-n-1}(m-n-2)!\\
&\leq &  C_{7}L^{m-3}(m-2)!\sum_{3\leq n\leq m- 2}\frac{m^2}{n^3(m-n)^2}\\
&\leq & C_{8}L^{m-3}(m-2)!.
\end{eqnarray*}
This along with  \reff{1} shows, in view of \reff{est f1 1},
\begin{eqnarray*}
  \sum_{1\leq n\leq m} {m\choose
       n}\norm{\p_q^n h^2_q}_{0,\mu}\norm{\p_q^{m-n}h_{pp}}_{0,\mu}
     \leq  \inner{C_{6}+C_{8}}L^{m-2}(m-2)!.
\end{eqnarray*}
Similarly,
we can find a constant $C_{9}$ such that
\begin{eqnarray*}
  2\sum_{1\leq n\leq m} {m\choose
       n}\norm{\p_q^n (h_ph_q)}_{0,\mu} \norm{\p_q^{m-n} h_{pq}}_{0,\mu}
 +\sum_{1\leq n\leq m} {m\choose
       n}\norm{\p_q^n h^2_p}_{0,\mu}  \norm{\p_q^{m-n}h_{qq}}_{0,\mu}\leq C_{9}L^{m-2}(m-2)!.
\end{eqnarray*}
Inserting the above two estimates into \reff{est f1}, we get the desired estimate
\reff{step2} by choosing $C_{2}=C_{6}+C_{8}+C_{9}$.

{\it Step 3)}~ We now prove
\bel{step3}
\norm{f_2}_{0,\mu}\leq  C_{3} L^{m-2}(m-2)!.
\enl
 In fact,
using  \reff{algebra} we have
   \begin{eqnarray}\label{est f4}
   \begin{split}
     \norm{f_2}_{0,\mu}&\leq \norm{ \beta}_{0,\mu} \norm{\p_q^m h^3_p}_{0,\mu}+g\norm{\rho}_{1,\mu}\sum_{0\leq n\leq m}{m\choose n}\norm{\p_q^n (h-d)}_{0,\mu}\norm{\p_q^{m-n}h_p^3}_{0,\mu}.
   \end{split}
 \end{eqnarray}
 Next we estimate the two terms on the right-hand side.

By the induction assumption \reff{nomind}, one has
\begin{eqnarray*}
\forall~3\leq j\leq m,\quad \norm{\p_q^j h_p}_{0,\mu}\leq \norm{\p_q^{j-1} h}_{2,\mu}\leq L^{j-2}(j-3)!.
\end{eqnarray*}
Then  using  Lemma \ref{stab},   with $\ell=2$,  $k_0=m$ ,  $H=L$,
$u_1=u_2=u_3=h_p$,  we conclude
\begin{eqnarray}\label{est f2 1}
  \norm{\p_q^m h_p^3}_{0,\mu}\leq C_{10}  L^{m-2}(m-3)!.
\end{eqnarray}
Similarly we have
\begin{eqnarray}\label{m-nhp3}
   \forall~0\leq n\leq m-3,\quad \norm{\p_q^{m-n} h_p^3}_{0,\mu}\leq C_{10}  L^{m-n-2}(m-n-3)!.
\end{eqnarray}
Write
\begin{eqnarray}\label{est f2 2}
  \begin{split}
  &\sum_{0\leq n\leq m}{m\choose n}\norm{\p_q^n (h-d)}_{0,\mu}\norm{\p_q^{m-n}h_p^3}_{0,\mu}\\
 & =\inner{\sum_{0\leq n\leq 3}+\sum_{4\leq n\leq m-3}+\sum_{m-2\leq n\leq m}}{m\choose n}\norm{\p_q^n (h-d)}_{0,\mu}\norm{\p_q^{m-n}h_p^3}_{0,\mu}.
  \end{split}
\end{eqnarray}
Using the induction assumption \reff{nomind} and \reff{m-nhp3},  we can compute directly to obtain
\begin{eqnarray}\label{est f2 a}
  \inner{\sum_{0\leq n\leq 3}+\sum_{m-2\leq n\leq m}}\norm{\p_q^n (h-d)}_{0,\mu}\norm{\p_q^{m-n}h_p^3}_{0,\mu}\leq C_{11}L^{m-2}(m-3)!
\end{eqnarray}
by choosing
\[
  C_{11}\geq C_{10}(9\norm{h}_{2,\mu}+d)+6\norm{h}_{2,\mu}\norm{h_p^3}_{2,\mu}+12\norm{h_p^3}_{2,\mu}.
\]
For the case when $4\leq n\leq m-3$, again by \reff{nomind} and \reff{m-nhp3} we have
\begin{eqnarray}\label{est f2 b}
\begin{split}
  &\sum_{4\leq n\leq m-3}{m\choose n}\norm{\p_q^n (h-d)}_{0,\mu}\norm{\p_q^{m-n}h_p^3}_{0,\mu}\\
  &\leq \sum_{4\leq n\leq m-3}{m\choose n}\norm{\p_q^{n-2} h}_{2,\mu}\norm{\p_q^{m-n}h_p^3}_{0,\mu}\\
  &\leq \sum_{4\leq n\leq m-3}\frac{m!}{n!(m-n)!}C_{10}L^{n-3}(n-4)!L^{m-n-2}(m-n-3)!\\
   &\leq C_{12}L^{m-5}(m-3)!\sum_{4\leq n\leq m-3}\frac{m^3}{n^4(m-n)^3}\\
   &\leq C_{13}L^{m-2}(m-3)!.
   \end{split}
\end{eqnarray}
Inserting \reff{est f2 a} and \reff{est f2 b} into \reff{est f2 2}, we get
\begin{eqnarray*}
  \sum_{0\leq n\leq m}{m\choose n}\norm{\p_q^n (h-d)}_{0,\mu}\norm{\p_q^{m-n}h_p^3}_{0,\mu}\leq (C_{11}+C_{13})L^{m-2}(m-3)!.
\end{eqnarray*}
Thus combining the above estimate and \reff{est f2 1}, we obtain, in view of \reff{est f4}, the desired estimate \reff{step3}, by choosing
$C_{3}=C_{10}\norm{\beta}_{0,\mu}+g\norm{\rho}_{1,\mu}(C_{11}+C_{13})+1$.

{\it Step 4)}~ Finally  we  prove
\bel{step4}
  \sum_{i=1}^2\norm{\varphi_i}_{0,\mu}\leq C_{4}L^{m-2}(m-2)!.
\enl

First for $\norm{\varphi_1}_{0,\mu}$, there exists a constant $C_{14}>0$ such that
\begin{eqnarray*}
  \norm{\varphi_1}_{0,\mu}\leq C_{14}L^{m-2}(m-2)!.
\end{eqnarray*}
The proof is similar as that of  \reff{step2} for $\norm{f_1}_{0,\mu}$, so we omit the details.
 
 Next for $\norm{\varphi_2}_{0,\mu}$,  by \reff{algebra}, we write
 \begin{eqnarray}\label{est varphi2}
   \begin{split}
   \norm{\varphi_2}_{0,\mu}&\leq 2\sigma \sum_{1\leq n\leq m}\norm{\p_q^{m-n} h_{qq}}_{0,\mu}\norm{\p_q^n\inner{ h_p^2{(1+h_q^2)^{-3/2}}}}_{0,\mu}.
   \end{split}
 \end{eqnarray}
 By the induction assumption \reff{nomind}, one has
 \begin{equation}\label{3}
  \forall~1\leq n\leq m-2,\quad \norm{\p_q^{m-n} h_{qq}}_{0,\mu}\leq\norm{\p_q^{m-n} h}_{2,\mu}\leq L^{m-n-1}(m-n-2)!,
 \end{equation}
 and
\begin{equation}\label{*}
\forall~3\leq n\leq m,\quad \norm{\p_q^n h_p}_{0,\mu}\leq \norm{\p_q^{n-1} h}_{2,\mu}\leq L^{n-2}(n-3)!.
\end{equation}
In view of \reff{*}, we use Lemma \ref{stab}  with $\ell=2$, $k_0=m$ ,  $H=L$,
$u_1=u_2=h_p$ and $u_3=1$, to conclude
\begin{equation}\label{4}
 \forall~3\leq n\leq m,\quad \norm{\p_q^{n} h^2_p}_{0,\mu}\leq C_{15} L^{n-2}(n-3)!.
\end{equation}

To estimate the norm $\norm{\p_q^n\inner{ h_p^2{(1+h_q^2)^{-3/2}}}}_{0,\mu}$, we use the following lemma and refer to \cite[Lemma 5.2]{LW} for its proof.

\begin{lemma}\label{+stab+}

Let $C_*\geq 1$  be the constant given in Lemma \ref{stab},   and let
$k_0\in\N$ with
$k_0\geq 3$.  Suppose  $\p_q^k u\in
C^{0,\mu}(\bar R)$ for any $k\leq k_0$.  If there
exist  two constants $C_0$ and $ \tilde H$ satisfying
\begin{equation}\label{biglittle}
  C_0\geq C_*\inner{2\big\|(1+u^2)^{-1}\big\|_{2,\mu}+2\big\|(1+u^2)^{-3/2}\big\|_{2,\mu}+\big\|\p_q(u^2)\big\|_{2,\mu}+1}^6
\end{equation}
and
\begin{equation}\label{+biglittle}
\tilde H\geq  2 C^2_0+\big\|\p_q^3\inner{(1+u^2)^{-1}}
 \big\|_{0,\mu}+\big\|\p_q^3\inner{(1+u^2)^{-3/2}}
 \big\|_{0,\mu},
\end{equation}
such that
\begin{equation}\label{+condition3}
 \forall~ 3\leq k\leq k_0,  \quad  \|\p_q^k (u^2)\|_{0,\mu}
  \leq C_0 \tilde H^{k-2}(k-3)!,
\end{equation}
then
\begin{eqnarray}\label{++conclusion}
\forall~ 3\leq k\leq k_0,  \quad\big\|\p_q^k \inner{(1+u^2)^{-3/2}} \big\|_{0,\mu}\leq C_0^2  \tilde H^{k-2}(k-3)!.
\end{eqnarray}

\end{lemma}

Since by the induction assumption \reff{nomind}, one has
\begin{equation*}
\forall~3\leq n\leq m,\quad \norm{\p_q^n h_q}_{0,\mu}\leq \norm{\p_q^{n-1} h}_{2,\mu}\leq L^{n-2}(n-3)!,
\end{equation*}
applying Lemma \ref{stab} with $\ell=2$, $u_1=u_2=h_q$, $k_0=m$ and $H=L$ gives 
\begin{eqnarray*}
  \forall~3\leq n\leq m,\quad \norm{\p_q^n h^2_q}_{0,\mu}\leq C_*(2\norm{h_q}_{2,\mu}+1)^2 L^{n-2}(n-3)!\leq C_0 L^{n-2}(n-3)!,
\end{eqnarray*}
where in the last inequality we chose
\begin{eqnarray}\label{deterC0}
  C_0\geq C_*\inner{2\big\|(1+h_q^2)^{-1}\big\|_{2,\mu}+2\big\|(1+h_q^2)^{-3/2}\big\|_{2,\mu}+\big\|\p_q(h_q^2)\big\|_{2,\mu}+2\norm{h_q}_{2,\mu}+1}^6.
\end{eqnarray}
If we choose $L$ large enough such that
\begin{eqnarray*}
  L\geq  2 C^2_0+\big\|\p_q^3\inner{(1+h_q^2)^{-1}}
 \big\|_{0,\mu}+\big\|\p_q^3\inner{(1+h_q^2)^{-3/2}}
 \big\|_{0,\mu},
\end{eqnarray*}
then we can use the above lemma \ref{+stab+}, with $u=h_q$, $k_0=m$ and $\tilde H=L$, to get
\begin{eqnarray}\label{parn3/2}
  \forall~3\leq n\leq m,  \quad\big\|\p_q^n \inner{(1+h_q^2)^{-3/2}} \big\|_{0,\mu}\leq C_0^2  L^{n-2}(n-3)!.
\end{eqnarray}
Now in view of \reff{4} and \reff{parn3/2}, applying Lemma \ref{stab}, with $\ell=2$, $u_1=C_{15}^{-1}h_p^2$ and $u_2=C_0^{-2}(1+h_q^2)^{-3/2}$, yields
\begin{eqnarray}\label{5a}
  \forall~3\leq n\leq m,  \quad\big\|\p_q^n \inner{h_p^2(1+h_q^2)^{-3/2}} \big\|_{0,\mu}\leq C_{16}L^{n-2}(n-3)!.
\end{eqnarray}

With \reff{3} and \reff{5a} in hand, we can write \reff{est varphi2} as
\begin{eqnarray*}
   \begin{split}
   \norm{\varphi_2}_{0,\mu}
   &\leq 2\sigma \inner{\sum_{1\leq n\leq 2}+\sum_{3\leq n\leq m-2}+\sum_{m-1\leq n\leq m}}\norm{\p_q^{m-n} h_{qq}}_{0,\mu}\norm{\p_q^n\inner{ h_p^2{(1+h_q^2)^{-3/2}}}}_{0,\mu},
   \end{split}
 \end{eqnarray*}
 and argue as in the previous steps to get
 \begin{eqnarray*}
   \begin{split}
   \norm{\varphi_2}_{0,\mu}&\leq C_{17}L^{m-2}(m-2)!.
   \end{split}
 \end{eqnarray*}
 Choosing $C_4=C_{14}+C_{17}$, we derive the estimate \reff{step4}.

Now we come back to the proof of Proposition \ref{propassu}.  Choose
$L$ in such a way that
\begin{equation*}
 L\geq\mathcal C\inner{C_1+ C_2+ C_3+ C_4}+2 C^2_0+\big\|\p_q^3\inner{(1+h_q^2)^{-1}}
 \big\|_{0,\mu}+\big\|\p_q^3\inner{(1+h_q^2)^{-3/2}}
 \big\|_{0,\mu}+\norm{\p_q^2 h}_{2,\mu}+1
\end{equation*}
with $\mathcal C, C_1,\cdots,C_4,C_0$ the constants given in \reff{par m h},
\reff{step1}, \reff{step2}, \reff{step3}, \reff{step4} and \reff{deterC0}.   Then
combining
\reff{par m h},  \reff{step1}, \reff{step2}, \reff{step3} and \reff{step4},  we have,
 \begin{eqnarray*}
    \norm{\p_q^m h}_{2,\mu}\leq \mathcal C\inner{C_1+ C_2+ C_3+ C_4} L^{m-2}(m-2)!\leq L^{m-1}(m-2)!.
 \end{eqnarray*}
 The validity of $(E_m)$ follows. Thus  the proof of
 Proposition \ref{propassu} is complete.
\end{proof}

\section{Gevrey regularity of the pseudo-stream function}\label{sec4}

Let $G^s\inner{[p_0,  0]}$,  $s\geq 1$,  be the Gevrey class; see
Definition \ref{def gevrey} of Gevrey function.   In
this section we assume  $\beta,\rho\in G^s\inner{[p_0,  0]}$.  Then by
the alternative characterization  of Gevrey function,   for any $p\in[p_0, 0]$ we can
find a neighborhood $U_p$ of $p$ and a constant $M_p$ such that
\[
\forall~k\in \mathbb N,\quad \sup_{t\in U_p\cap [p_0,0]} \abs{\p_p^k
  \beta(t)}\leq M_p^{k+1}(k!)^s.
\]
Note $[p_0,0]$ is compact in $\mathbb R$; this allows us to find a constant $M$ such
that
\begin{equation}\label{GevCon}
\forall~k\in \mathbb N,\quad \sup_{p\in[p_0, 0]}\abs{\p_p^k
  \beta(p)}\leq  M^{k+1}(k!)^s.
\end{equation}
Similarly we can find a constant $N$ such that
\begin{equation}\label{Gev rho}
\forall~k\in \mathbb N,\quad \sup_{p\in [p_0,0]} \abs{\p_p^k
  \rho(p)}\leq N^{k+1}(k!)^s.
\end{equation}

We prove now  Gevrey regularity of the pseudo-stream function,
i.e., Theorem \ref{thm2}. In view of Proposition \ref{reserv},  it suffices to
show the following result for the height function $h(q,p)$.

\begin{proposition}\label{prophpq}
 Let $\beta,\rho\in G^s\inner{[p_0,  0]}$ with $s\geq 1$,  and let $h\in C^{2,\mu}(\bar
 R)$ be a solution to \reff{Equh1}-\reff{Equh3}.
Then there exist two  constants $ L_1, L_2$ with $L_2\geq L_1\geq 1$ ,  such
that for any $m\geq 2$  we have the
following estimate
   \begin{eqnarray*}
     (F_m):\qquad \forall~\alpha=(\alpha_1,\alpha_2)\in\N^2,~ \abs\alpha=m,\quad  \norm{\p^{\alpha}h}_{2}\leq
     L_1^{\alpha_1-1}L_2^{\alpha_2} [(\abs\alpha-2)!]^s.
   \end{eqnarray*}
Recall $\norm{\cdot}_2$ stands for the  H\"{o}lder norm $\norm{\cdot}_{C^{2,0}(\bar R)}$. Thus $h\in
 G^s(\bar R)$;  in particular if $s=1$ then $h$ is analytic
 in $\bar R$.
\end{proposition}

\begin{remark}
 As to be seen in the proof, the constants $L_1, L_2$ depend on  the
 constant $L$  given in Proposition \ref{propassu} and the constants $M,N$
 in \reff{GevCon} and \reff{Gev rho},  but
 independent of the order $m$ of derivative.
\end{remark}

\begin{remark}\label{+remReg}
Note   $\beta,\rho\in G^s([p_0,0])\subset C^\infty([p_0,0])$.  By  Remark
\ref{remReg}   we see $\p_q h \in C^{2,\mu}(\bar R)$.  Then
differentiating the
   equation \reff{Equh1} with respect to $p$,    we can  obtain $h\in
   C^{3, \mu}(\bar R)$;  see
   \cite{MR2027299} for details.    Repeating  this procedure gives  $ h\in
  C^{k,\mu}(\bar R)$  for any $k \in \mathbb N$,  since $\beta,\rho\in C^\infty([p_0,0])$.
\end{remark}

To confirm the last statement in the above proposition \ref{prophpq},
we   choose $C$ in such a way
that
\[
     C=\max\set{L_1, L_2,  \norm{h}_{1,\mu}},
\]
which,  along with the estimate $(F_m)$ with $m\geq 2$   in Proposition \ref{prophpq},  yields
\begin{eqnarray*}
\forall~\alpha\in\mathbb N^2, \quad \max_{(q,p)\in
  \bar R}|\partial^\alpha{h}(q,p)| \leq C^{\abs\alpha+1}(\abs \alpha!)^s .
\end{eqnarray*}
This gives $h\in
 G^s(\bar R)$.

In order to prove   Proposition \ref{prophpq}, we need the following
technical lemma and will present  its proof at the end of this section.

\begin{lemma}\label{+stab}

Let $s\geq 1$, and $H_1$ and $H_2$ be two  constants with $H_2\geq
H_1\geq 1$. Suppose that $\alpha_0$ is a given multi-index with $\abs{\alpha_0}\geq3$, and $u,v,w\in
C^{\abs{\alpha_0},\mu}(\bar R)$.  For $j=0,1,2$, denote
\begin{equation*}
 \mathcal{A}_{j}=\set{~f\in C^{|\alpha_0|,\mu}(\bar R)~\big|~ \forall~\alpha=(\alpha_1,\alpha_2)\leq\alpha_0,~ \abs{\alpha}\geq j+1,  \quad  \|\p^\alpha f\|_0
  \leq H_1^{\alpha_1-j}H_2^{\alpha_2}[(\abs\alpha-j-1)!]^s}.
\end{equation*}
Then there exists a constant $c_*$, depending only on the $C^{2,0}$-norms
of $u, v$ and $w$, but independent of $\alpha_0$,  such that
\begin{enumerate}[(a)]
  \item\label{conc1} if $u\in \mathcal{A}_2$ and $v\in \mathcal{A}_1$,
    then $c_*^{-1}uv\in \mathcal{A}_1$, that is,
  \[
   \forall~\alpha=(\alpha_1,\alpha_2)\leq\alpha_0,~ \abs{\alpha}\geq 2,  \quad \big\|\p^\alpha\inner{uv}\big\|_0\leq c_* H_1^{\alpha_1-1}H_2^{\alpha_2}[(\abs\alpha-2)!]^s;
  \]
 if additionally  $w\in \mathcal A_2$ then $c_*^{-1}uvw\in
 \mathcal{A}_1$;
  \item\label{conc2}  if $u\in \mathcal{A}_2$ and $v,w\in \mathcal{A}_1$, then $c_*^{-1}uvw\in \mathcal{A}_0$;
  \item\label{conc3}  if $u,v\in \mathcal{A}_2$ and $w\in
    \mathcal{A}_0$, then $c_*^{-1}uvw\in \mathcal{A}_0$;
 \item\label{conc4}  if $u, v, w\in \mathcal{A}_2$, then $c_*^{-1}H_1
   uvw\in \mathcal{A}_1$, that is,
\[
   \forall~\alpha=(\alpha_1,\alpha_2)\leq\alpha_0,~ \abs{\alpha}\geq 2,  \quad \big\|\p^\alpha\inner{uvw}\big\|_0\leq c_* H_1^{\alpha_1-2}H_2^{\alpha_2}[(\abs\alpha-2)!]^s;
  \]
  \item \label{conc5}if $u\in \mathcal{A}_3$ and $v\in \mathcal{A}_2$,
    then $\tilde c_*^{-1}uv\in \mathcal{A}_2$, with $\tilde c_*$ a constant depending  on the $C^{3,0}$-norm of $u$ and $C^{2,0}$-norm of $v$.
\end{enumerate}

\end{lemma}

Now we prove Proposition \ref{prophpq}.

\begin{proof}[Proof of Proposition \ref{prophpq}]
In view of Remark \ref{+remReg} we may assume  that  $ h\in
  C^{k,\mu}(\bar R)$ for any $k\in\N$.  We now use induction on $m$ to
  prove the estimate $(F_m)$.  First for $m=2$,
   $(F_m)$ obviously holds by  choosing  $ L_1, L_2$ in such a way
   that
 \begin{equation}\label{l1l21}
   L_2\geq L_1\geq \norm{h}_{4}+1.
\end{equation}
Next let $m\geq 3$  and assume that
   $(F_j)$ holds for any $j$ with $2\leq j\leq m-1$, that is,
\begin{equation}\label{fm1}
\forall~\gamma=(\gamma_1,\gamma_2),~ 2\leq
\gamma_1+\gamma_2\leq m-1, \quad \norm{\p^{\gamma}h}_{2}\leq L_1^{\gamma_1-1}
L_1^{\gamma_2}[(\abs\gamma-2)!]^s.
\end{equation}
We have to prove the validity of $(F_m)$.  This is equivalent to show
the following estimate
\begin{equation}\label{fmn}
(F_{m,n}): \qquad\norm{\p_q^{m-n}\p_p^{n}h}_{2}\leq L_1^{m-n-1}
L_2^{n}[(m-2)!]^s
\end{equation}
holds for all $n$ with $0\leq n\leq m.$

In what follows we use induction on $n$ to show \reff{fmn} with fixed
$m\geq 3$.
Firstly note that $s\geq 1$, and thus from Proposition \ref{propassu} we see that $(F_{m,0})$ holds
   if we choose
\begin{equation}\label{l1l22}
L_1\geq L.
\end{equation}   Next let $1\leq n\leq m$ and  assume that $(F_{m,i})$
holds for all $i$ with $0\leq
   i\leq n-1$, that is,
\begin{equation}\label{fn1}
\forall~0\leq i\leq n-1,\quad \norm{\p_q^{m-i}\p_p^i h}_{2}\leq L_1^{m-i-1}
L_2^{i}[(m-2)!]^s.
\end{equation}
We have to show $(F_{m,n})$ holds as well, i.e., to prove that
   \begin{equation}\label{hqmn}
     \norm{\p_q^{m-n}\p_p^n h}_{2}\leq L_1^{m-n-1}L_2^{n}[(m-2)!]^s.
   \end{equation}
To do so, we firstly compute, with $1\leq n\leq m$,
   \begin{eqnarray*}
     \norm{\p_q^{m-n}\p_p^n h}_{2}&\leq&\norm{\p_q^{m-n}\p_p^n h}_{1}+\norm{\p_q^{m-n+2}\p_p^{n} h}_{0}+\norm{\p_q^{m-n+1}\p_p^{n+1} h}_{0}+\norm{\p_q^{m-n}\p_p^{n+2} h}_{0}\\&\leq&\norm{\p_q^{m-n}\p_p^{n-1} h}_{2}+2\norm{\p_q^{m-(n-1)}\p_p^{n-1} h}_{2}+\norm{\p_q^{m-n}\p_p^{n+2} h}_{0}.
   \end{eqnarray*}
  The induction assumptions \reff{fm1} and \reff{fn1} yield
   \begin{eqnarray*}
     \norm{\p_q^{m-n}\p_p^{n-1} h}_{2}+2\norm{\p_q^{m-(n-1)}\p_p^{n-1} h}_{2}
       &\leq& L_1^{m-n-1}L_2^{n-1}[(m-3)!]^s+2L_1^{m-n}L_2^{n-1}[(m-2)!]^s\\
       &\leq& L_2^{-1}(1+2L_1) L_1^{m-n-1}L_2^{n}[(m-2)!]^s\\
       &\leq& \frac{1}{2}L_1^{m-n-1}L_2^{n}[(m-2)!]^s,
   \end{eqnarray*}
   where in the last inequality we choose
   \begin{equation}\label{l1l23}
       L_2\geq 8L_1\geq 8.
\end{equation}
 Accordingly, in order to obtain
   \reff{hqmn}, it  suffices to prove
\begin{equation}\label{+hqmn}
     \norm{\p_q^{m-n}\p_p^{n+2} h}_{0}\leq  \frac{1}{2} L_1^{m-n-1}L_2^{n}[(m-2)!]^s.
   \end{equation}
The rest is occupied by the proof of the above estimate.

From now on we fix $m$  and $n$ with $m\geq 3$ and $1\leq n\leq m$,  and denote $\alpha=(\alpha_1,\alpha_2)=(m-n,n)$.  Applying $\p^{\alpha}=\p_q^{m-n}\p_p^n$ on both sides of the equation \reff{Equh1} gives
   \begin{eqnarray*}
   \begin{split}
     (1+h_q^2)(\p^\alpha h_{pp})=&-\sum_{\gamma\leq \alpha,\gamma\neq 0}{\alpha\choose\gamma}(\p^\gamma h_q^2)(\p^{\alpha-\gamma}h_{pp})+2\p^{\alpha}(h_p h_q h_{qp})-\p^{\alpha}(h_p^2h_{qq})\\
     &+g\p^{\alpha}\inner{(h-d)\rho_ph_p^3}-\p^{\alpha}\inner{\beta h_p^3},
     \end{split}
   \end{eqnarray*}
   which implies
   \begin{eqnarray*}
     \norm{(1+h_q^2)(\p^\alpha h_{pp})}_{0}&\leq& \sum_{\gamma\leq \alpha,\gamma\neq 0}{\alpha\choose\gamma}\norm{\p^\gamma h_q^2}_{0}\norm{\p^{\alpha-\gamma}h_{pp}}_{0}+2\norm{\p^{\alpha}(h_p h_q h_{qp})}_{0}\\[3pt]
     &&+\norm{\p^{\alpha}(h_p^2h_{qq})}_{0}+g\norm{(h-d)\rho_ph_p^3}_0+\norm{\p^{\alpha}\inner{\beta h_p^3}}_{0}.
   \end{eqnarray*}
   Since
   \begin{equation}\label{9}
     \norm{\p_q^{m-n}\p_p^{n+2} h}_{0}=\norm{\p^\alpha h_{pp}}_{0}\leq\norm{(1+h_q^2)(\p^\alpha h_{pp})}_{0},
   \end{equation}
we obtain,  with  $\alpha=(\alpha_1,\alpha_2)=(m-n,n)$,
\begin{equation}\label{5}
    \begin{split}
     \norm{\p_q^{m-n}\p_p^{n+2} h}_{0}\leq& \sum_{\gamma\leq \alpha,\gamma\neq 0}{\alpha\choose\gamma}\norm{\p^\gamma h_q^2}_{0}\norm{\p^{\alpha-\gamma}h_{pp}}_{0}+2\norm{\p^{\alpha}(h_p h_q h_{qp})}_{0}\\
     &+\norm{\p^{\alpha}(h_p^2h_{qq})}_{0}+g\norm{(h-d)\rho_ph_p^3}_0+\norm{\p^{\alpha}\inner{\beta h_p^3}}_{0}.
    \end{split}
   \end{equation}
We now treat the  terms on the right-hand side through the
following lemmas.

To simplify the notations,  we will use $c_j, j\geq 1$, to denote
suitable {\it harmless constants} larger than $1$.  By harmless constants it
means that these constants are independent of
$m$ and $n$. The following three lemmas, Lemma \ref{lem401}--Lemma \ref{lem403}, have been proved in \cite{LW}. For completeness, we present their proofs here again.

 \begin{lemma}\label{lem401}
   For $\alpha=(\alpha_1,\alpha_2)=(m-n,n)$ with
     $1\leq n\leq m$,  we have
      \begin{equation*}
       \sum_{\gamma\leq \alpha,\gamma\neq
        0}{\alpha\choose\gamma}\norm{\p^\gamma
        h_q^2}_{0}\norm{\p^{\alpha-\gamma}h_{pp}}_{0}\leq c_1 L_1^{\alpha_1-2}L_2^{\alpha_2} [(\abs{\alpha}-2)!]^{s}.
    \end{equation*}
    \end{lemma}

\begin{proof}[Proof of the lemma]
      We firstly use Lemma \ref{+stab} to  treat the term
      $\norm{\p^\gamma h_q^2}_0$ with $3\leq \abs\gamma\leq
      \abs\alpha=m$.   To do so, write
      $\gamma=\tilde\gamma+(\gamma-\tilde\gamma)$ with
      $|\tilde\gamma|=\abs\gamma-1\geq 2$.  Without loss of  generality we may
      take $\gamma-\tilde\gamma=(0,1)$, and the arguments below also
      holds when $\gamma-\tilde\gamma=(1,0)$.   Thus
  \begin{equation}\label{rel}
  \p^\gamma h_q^2=2\p^{\tilde \gamma} \inner{h_q h_{qp}}.
\end{equation}
Note that for any
      $\xi=(\xi_1,\xi_2)\leq \tilde\gamma$ with
      $|\xi|\geq 3$, we have, using the induction assumption
      \reff{fm1},
\[
  \norm{\p^{\xi} h_{qp}}_{0}\leq \norm{\p^{\xi} h}_{2}\leq L_1^{\xi_1-1}L_2^{\xi_2}
[(|\xi|-2)!]^s,
\]
and
    \begin{equation}\label{10}
      \norm{\p^{\xi} h_q}_{0}\leq
\left\{
\begin{array}{lll}
\norm{\p_q^{\xi_1}\p_p^{\xi_2-1} h}_{2} \leq
L_1^{\xi_1-1} L_2^{\xi_2-1}
[(|\xi|-3)!]^s\leq L_1^{\xi_1-2}L_2^{\xi_2}
[(|\xi|-3)!]^s,\quad &\xi_2\geq 1,\\[3pt]
\norm{\p_q^{\xi_1-1}h}_{2} \leq
L_1^{\xi_1-2}
[(|\xi|-3)!]^s= L_1^{\xi_1-2}L_2^{\xi_2}
[(|\xi|-3)!]^s,\quad &\xi_2=0,
\end{array}
    \right.
    \end{equation}
    where in the case $\xi_2\geq 1$ we used $L_2\geq L_1$.
 Therefore applying Lemma \ref{+stab}-\reff{conc1}, with  $H_1=L_1$, $H_2=L_2$,  $u=h_q$ and
 $v=h_{qp}$,  gives
\[
  \norm{\p^{\tilde \gamma} (h_qh_{qp})}_{0}\leq c_5
      L_1^{\tilde \gamma_1-1}L_2^{\tilde \gamma_2} [(|\tilde
      \gamma|-2)!]^s=  c_5
      L_1^{ \gamma_1-1}L_2^{\gamma_2-1} [(|
      \gamma|-3)!]^s\leq  c_5
      L_1^{ \gamma_1-2}L_2^{\gamma_2} [(|
      \gamma|-3)!]^s,
\]
the last inequality holding because $L_2\geq
L_1$. This along with the relation \reff{rel} yields
    \begin{equation}\label{6}
      \forall ~\gamma, ~3\leq \abs{\gamma}\leq
      \abs\alpha,\quad\norm{\p^{\gamma} h^2_q}_{0}\leq 2 c_5
      L_1^{\gamma_1-2}L_2^{\gamma_2} [(|\gamma|-3)!]^s.
    \end{equation}
    On the other hand, for the term
    $\norm{\p^{\alpha-\gamma}h_{pp}}_{0}$, we have,  by the induction
  assumption \reff{fm1},
    \begin{equation}\label{7}
      \forall~\gamma\leq \alpha,~1\leq\abs{\gamma}\leq\abs{\alpha}-2,\quad \norm{\p^{\alpha-\gamma}h_{pp}}_{0}\leq
      \norm{\p^{\alpha-\gamma}h}_{2}\leq
      L_1^{\alpha_1-\gamma_1-1}L_2^{\alpha_2-\gamma_2}[(\abs{\alpha}-\abs{\gamma}-2)!]^s.
    \end{equation}

    Next we write
    \begin{equation*}
     \begin{split}
      \sum_{\gamma\leq \alpha,\gamma\neq 0}{\alpha\choose\gamma}\norm{\p^\gamma h_q^2}_{0}\norm{\p^{\alpha-\gamma}h_{pp}}_{0}&=\inner{\sum_{\stackrel{\gamma\leq \alpha}{1\leq\abs{\gamma}\leq 2}}+\sum_{\stackrel{\gamma\leq \alpha}{3\leq\abs{\gamma}\leq \abs{\alpha}-2}}+\sum_{\stackrel{\gamma\leq \alpha}{\abs{\gamma}\geq \abs{\alpha}-1}}}{\alpha\choose\gamma}\norm{\p^\gamma h_q^2}_{0}\norm{\p^{\alpha-\gamma}h_{pp}}_{0}\\
      &=J_1+J_2+J_3.
      \end{split}
    \end{equation*}
    By virtue of \reff{6} and \reff{7},  direct computation as in
    \reff{1},  shows that
    \begin{equation*}
      J_1+J_3\leq c_6L_1^{\alpha_1-2}L_2^{\alpha_2}[(\abs{\alpha}-2)!]^s.
    \end{equation*}
 Next for $J_2$, which appears only when $\abs{\alpha}\geq 5$, we have by \reff{6} and \reff{7} that
    \begin{equation*}
      \begin{split}
        J_2&\leq 2c_5\sum_{\stackrel{\gamma\leq
            \alpha}{3\leq\abs{\gamma}\leq
            \abs{\alpha}-2}}\frac{\abs{\alpha}!}{\abs{\gamma}!(\abs{\alpha}-\abs{\gamma})!}
             L_1^{\gamma_1-2}L_2^{\gamma_2} [(|\gamma|-3)!]^sL_1^{\alpha_1-\gamma_1-1}L_2^{\alpha_2-\gamma_2}[(\abs{\alpha}-\abs{\gamma}-2)!]^s\\
      &\leq c_7\sum_{\stackrel{\gamma\leq \alpha}{3\leq\abs{\gamma}\leq \abs{\alpha}-2}}\frac{\abs{\alpha}!}{\abs{\gamma}^3(\abs{\alpha}-\abs{\gamma})^2}
     L_1^{\alpha_1-3}L_2^{\alpha_2} [(|\gamma|-3)!]^{s-1}[(\abs{\alpha}-\abs{\gamma}-2)!]^{s-1}\\
     &\leq  c_7L_1^{\alpha_1-3}L_2^{\alpha_2}\sum_{\stackrel{\gamma\leq \alpha}{3\leq\abs{\gamma}\leq \abs{\alpha}-2}}\frac{\abs{\alpha}!}{\abs{\gamma}^3(\abs{\alpha}-\abs{\gamma})^2}
      [(\abs{\alpha}-5)!]^{s-1}\\
      &\leq c_7L_1^{\alpha_1-3}L_2^{\alpha_2} [(\abs{\alpha}-2)!]^{s}\sum_{\stackrel{\gamma\leq \alpha}{3\leq\abs{\gamma}\leq \abs{\alpha}-2}}\frac{\abs{\alpha}^2}{\abs{\gamma}^3(\abs{\alpha}-\abs{\gamma})^2}\\
      &\leq c_8L_1^{\alpha_1-3}L_2^{\alpha_2} [(\abs{\alpha}-2)!]^{s},
      \end{split}
    \end{equation*}
    the last inequality using Lemma \ref{lemSum}. Therefore, choosing $c_1=c_6+c_8$,   we can
    combine the estimates for $J_1, J_2$ and $ J_3$ to complete the proof of the lemma.
    \end{proof}

 \begin{lemma}
  For $\alpha=(\alpha_1,\alpha_2)=(m-n,n)$ with
     $1\leq n\leq m$, we have
      \begin{equation*}
        2\norm{\p^{\alpha}(h_p h_q
          h_{qp})}_{0}+\norm{\p^{\alpha}(h_p^2h_{qq})}_{0}\leq c_2 L_1^{\alpha_1}L_2^{\alpha_2-1} [(\abs{\alpha}-2)!]^{s}.
      \end{equation*}
    \end{lemma}

    \begin{proof}[Proof of the lemma]
    Since $n\geq 1$, we can write $\alpha=\tilde\alpha+(0,1)$ with $\tilde \alpha=(\tilde\alpha_1,\tilde\alpha_2)=(m-n,n-1)$.
    Thus
    \begin{equation}\label{d}
      \p^{\alpha}(h_p h_q h_{qp})=\p^{\tilde\alpha}\inner{h_{pp} h_q h_{qp}+h_p h_{qp} h_{qp}+h_p h_q h_{qpp}}.
    \end{equation}
    We next compute the estimate for the term $\p^{\tilde\alpha}(h_{pp} h_q h_{qp})$.
      For any $\gamma\leq \tilde\alpha$ with $\abs{\gamma}\geq 3$,  we
      have, as for $\norm{\p^{\xi} h_q}_{0}$ in \reff{10},
      \begin{equation*}
        \norm{\p^\gamma h_q}_0\leq L_1^{\gamma_1-2}L_2^{\gamma_2}[(\abs{\gamma}-3)!]^s,\qquad \norm{\p^\gamma h_p}_0\leq L_1^{\gamma_1-2}L_2^{\gamma_2}[(\abs{\gamma}-3)!]^s,
      \end{equation*}
      and by  the induction assumption  \reff{fm1} and \reff{fn1}, in view of
      $\gamma_2\leq \tilde\alpha_2= n-1$,
\begin{eqnarray*}
\norm{\p^\gamma h_{pp}}_0&\leq&\norm{\p^\gamma h}_2\leq
L_1^{\gamma_1-1}L_2^{\gamma_2}[(\abs{\gamma}-2)!]^s,\\
\norm{\p^\gamma h_{qp}}_0&\leq &\norm{\p^\gamma h}_2\leq
L_1^{\gamma_1-1}L_2^{\gamma_2}[(\abs{\gamma}-2)!]^s,\\
\norm{\p^\gamma h_{qpp}}_0&\leq &\norm{\p_q^{\gamma_1+1}\p_p^{\gamma_2}
  h}_2 \leq L_1^{\gamma_1}L_2^{\gamma_2}[(\abs{\gamma}-1)!]^s.
\end{eqnarray*}
Thus   we obtain,
using Lemma
  \ref{+stab}-\reff{conc1} with $u=h_{p}$, $w=h_{p}$ and $v=h_{qp}$,
\begin{eqnarray*}\label{a}
        \norm{\p^{\tilde\alpha}(h_{p} h_q h_{qp})}_0\leq c_9 L_1^{\tilde\alpha_1-1}L_2^{\tilde\alpha_2}[(\abs{\tilde\alpha}-2)!]^s=c_9 L_1^{\alpha_1-1}L_2^{\alpha_2-1}[(\abs{\alpha}-3)!]^s
        \leq c_9 L_1^{\alpha_1}L_2^{\alpha_2-1}[(\abs{\alpha}-2)!]^s.
      \end{eqnarray*}
Similarly,  using Lemma
  \ref{+stab}-\reff{conc2} with $u=h_{p}$, $v=w=h_{qp}$, gives
 \begin{eqnarray*}\label{a}
        \norm{\p^{\tilde\alpha}(h_{p} h_{qp} h_{qp})}_0\leq c_{10} L_1^{\tilde\alpha_1}L_2^{\tilde\alpha_2}[(\abs{\tilde\alpha}-1)!]^s
        =c_{10} L_1^{\alpha_1}L_2^{\alpha_2-1}[(\abs{\alpha}-2)!]^s,
      \end{eqnarray*}
while  using Lemma
  \ref{+stab}-\reff{conc3} with $u=h_{p}$, $v=h_{q}$ and $w=h_{qpp}$ gives
 \begin{eqnarray*}\label{a}
       \norm{\p^{\tilde\alpha}(h_{p} h_{p} h_{qpp})}_0\leq c_{11} L_1^{\tilde\alpha_1}L_2^{\tilde\alpha_2}[(\abs{\tilde\alpha}-1)!]^s
        =c_{11} L_1^{\alpha_1}L_2^{\alpha_2-1}[(\abs{\alpha}-2)!]^s.
      \end{eqnarray*}
Combining the above inequalities, we have,  in view of \reff{d},
      \begin{eqnarray*}
       2 \norm{\p^{\alpha}(h_p h_q h_{qp})}_0\leq 2(c_9+c_{10}+c_{11}) L_1^{\alpha_1}L_2^{\alpha_2-1}[(\abs{\alpha}-2)!]^s.
      \end{eqnarray*}
   The treatment for the term $\norm{\p^{\alpha}(h_p^2h_{qq})}_{0}$ is
   completely the same as above, so we have
      \begin{eqnarray*}
        \norm{\p^{\alpha}(h_p^2h_{qq})}_{0}\leq c_{12} L_1^{\alpha_1}L_2^{\alpha_2-1}[(\abs{\alpha}-2)!]^s.
      \end{eqnarray*}
      Combining the above two estimates, we choose
      $c_2=2(c_9+c_{10}+c_{11})+c_{12}$ to complete the proof of the lemma.
    \end{proof}

   \begin{lemma}\label{lem403}
     Let $\beta\in G^s([p_0,0])$ and $L_2\geq L_1\tilde M$ with $\tilde M$ a constant depending on $M$ given in \reff{GevCon} and $s$. We have, for $\alpha=(\alpha_1,\alpha_2)=(m-n,n)$ with
     $1\leq n\leq m$,
 \begin{equation*}
        \norm{\p^{\alpha}\inner{\beta h_p^3}}_{0}\leq c_3 L_1^{\alpha_1-2}L_2^{\alpha_2} [(\abs{\alpha}-2)!]^{s}.
      \end{equation*}
    \end{lemma}

 \begin{proof}[Proof of the lemma]
As for $\norm{\p^{\xi} h_q}_{0}$ in \reff{10},   we have by induction
      \begin{equation*}
        \forall~  \gamma\leq \alpha,~\abs{\gamma}\geq 3,\quad \norm{\p^{\gamma} h_p}_0\leq L_1^{\gamma_1-2}L_2^{\gamma_2}\com{(|\gamma|-3)!}^s.
      \end{equation*}
Thus using Lemma \ref{+stab}-\reff{conc4} with $H_i=L_i, i=1,2$,
$u=v=w=h_p$,  we deduce that $c_{13}^{-1}L_1 h_p^3\in\mathcal A_1$, that is,
      \begin{equation}\label{betahp3}
        \forall~\gamma\leq \alpha,~\abs{\gamma}\geq 2,\quad
        \norm{\p^\gamma \inner{ h_p^3}}_0 \leq c_{13} L_1^{\gamma_1-2}L_2^{\gamma_2}[(\abs{\gamma}-2)!]^s.
      \end{equation}
On the other hand, since  $\beta(p)\in G^s([p_0, 0])$,  then
using  \reff{GevCon}
 gives
\begin{eqnarray*}
   \forall~\abs{\gamma}\geq 3,\quad \ \norm{\p^\gamma \beta}_0\leq
\left\{
\begin{array}{lll}
0,&\quad  \gamma_1\geq 1,\\
 \norm{\p_p^{\gamma_2} \beta}_0 \leq M^{\gamma_2+1}(\gamma_2!)^s \leq
 \tilde M^{\gamma_2}[(\gamma_2-3)!]^s,
 &\quad \gamma_1=0,
\end{array}
\right.
\end{eqnarray*}
where $\tilde M$ in the last inequality is a constant
depending only on  $M$ and $s$.  Thus if we choose $L_1, L_2$ in such a
way that
\begin{equation}\label{l1l24}
L_2\geq L_1 \tilde M,
\end{equation}
then we have
 \begin{equation}\label{gambeta3}
 \forall~\gamma\leq \alpha,~\abs{\gamma}\geq 3,\quad\norm{\p^\gamma \beta}_0\leq  L_1^{\gamma_1-2}L_2^{\gamma_2}[(\abs{\gamma}-3)!]^s.
\end{equation}
Now we write
\begin{eqnarray*}
 \norm{\p^{ \alpha}( \beta h_p^3)}_0\leq \sum_{\gamma\leq
   \alpha}\frac{\abs\alpha !}{\abs\gamma!(\abs\alpha-\abs\gamma)!}\norm{\p^{ \gamma}\beta }_0 \norm{\p^{ \alpha-\gamma}(  h_p^3)}_0.
\end{eqnarray*}
This together with \reff{betahp3} and \reff{gambeta3} allows us to  argue as the
treatment of $J_1-J_3$ in Lemma \ref{lem401},  to conclude
\begin{eqnarray*}
 \norm{\p^{ \alpha}( \beta h_p^3)}_0\leq c_{14}
 L_1^{\alpha_1-2}L_2^{\alpha_2}[(|\alpha|-2)!]^s.
\end{eqnarray*}
Thus the desired estimate follows  if choosing $c_3=c_{14}$.  The proof
is thus complete.
\end{proof}

\begin{lemma}\label{lem404}
     Let $\rho\in G^s([p_0,0])$. We have, for $\alpha=(\alpha_1,\alpha_2)=(m-n,n)$ with
     $1\leq n\leq m$,
 \begin{equation*}
        g\norm{\p^{\alpha}\inner{(h-d)\rho_p h_p^3}}_{0}\leq c_4 L_1^{\alpha_1-2}L_2^{\alpha_2} [(\abs{\alpha}-2)!]^{s}.
      \end{equation*}
    \end{lemma}

\begin{proof}[Proof of the lemma]
As in the previous lemma,   we have
      \begin{equation}\label{101}
        \forall~\gamma\leq \alpha,~\abs{\gamma}\geq 2,\quad
        \norm{\p^\gamma \inner{ h_p^3}}_0 \leq c_{13} L_1^{\gamma_1-2}L_2^{\gamma_2}[(\abs{\gamma}-2)!]^s.
      \end{equation}
      Next as the treatment for $\norm{\p^{\xi} h_q}_{0}$ in \reff{10}, we have
      \begin{eqnarray}\label{99}
        \forall~\gamma\leq \alpha,~\abs{\gamma}\geq 4,\quad \norm{\p^\gamma (h-d)}_0=\norm{\p^\gamma h}_0\leq L_1^{\gamma_1-3}L_2^{\gamma_2}[(\abs{\gamma}-4)!]^s.
      \end{eqnarray}
Moreover, since  $\rho(p)\in G^s([p_0, 0])$,  then
using  \reff{Gev rho}
 gives
\begin{eqnarray*}
   \forall~\abs{\gamma}\geq 3,\quad \ \norm{\p^\gamma \rho_p}_0\leq
\left\{
\begin{array}{lll}
0,&\quad  \gamma_1\geq 1,\\
 \norm{\p_p^{\gamma_2+1} \rho}_0 \leq N^{\gamma_2+2}[(\gamma_2+1)!]^s \leq
 \tilde N^{\gamma_2}[(\gamma_2-3)!]^s,
 &\quad \gamma_1=0,
\end{array}
\right.
\end{eqnarray*}
with $\tilde N$ a constant depending only on $N$ and $s$.
Thus if we choose $L_1, L_2$ in such a
way that
\begin{equation}\label{l1l25}
L_2\geq L_1 \tilde N,
\end{equation}
then we have
 \begin{equation}\label{gam rhop}
 \forall~\gamma\leq \alpha,~\abs{\gamma}\geq 3,\quad\norm{\p^\gamma \rho_p}_0\leq  L_1^{\gamma_1-2}L_2^{\gamma_2}[(\abs{\gamma}-3)!]^s.
\end{equation}
Combining \reff{99} and \reff{gam rhop}, we can apply Lemma \ref{+stab}-\reff{conc5}, with $u=h-d$ and $v=\rho_p$, to conclude,
\begin{eqnarray}\label{98}
  \forall~\gamma\leq \alpha,~\abs{\gamma}\geq 3,\quad\norm{\p^\gamma \inner{(h-d)\rho_p}}_0\leq c_{15} L_1^{\gamma_1-2}L_2^{\gamma_2}[(\abs{\gamma}-3)!]^s.
\end{eqnarray}

Now we write
\begin{eqnarray*}
\begin{split}
 &g\norm{\p^{\alpha}\inner{(h-d)\rho_p h_p^3}}_{0}\leq g\sum_{\abs\gamma\leq
   \abs\alpha}\frac{\abs\alpha !}{\abs\gamma!(\abs\alpha-\abs\gamma)!}\norm{\p^{ \gamma}(h-d)\rho_p }_0 \norm{\p^{ \alpha-\gamma} h_p^3}_0\\
  & \leq g\inner{\sum_{\stackrel{\gamma\leq\alpha}{0\leq\abs\gamma\leq2}}
   +\sum_{\stackrel{\gamma\leq\alpha}{3\leq\abs\gamma\leq\abs{\alpha}-2}}+\sum_{\stackrel{\gamma\leq\alpha}{\abs\alpha-1\leq\abs\gamma}}}\frac{\abs\alpha !}{\abs\gamma!(\abs\alpha-\abs\gamma)!}\norm{\p^{ \gamma}(h-d)\rho_p }_0 \norm{\p^{ \alpha-\gamma} h_p^3}_0.
   \end{split}
\end{eqnarray*}
Again arguing as the
treatment of $J_1-J_3$ in Lemma \ref{lem401}, we use \reff{101} and \reff{98} to conclude
\begin{eqnarray*}
  g\norm{\p^{\alpha}\inner{(h-d)\rho_p h_p^3}}_{0}\leq  c_4 L_1^{\alpha_1-2}L_2^{\alpha_2}[(\abs{\alpha}-2)!]^s,
\end{eqnarray*}
completing the proof of the lemma.
\end{proof}

We now continue the proof of Proposition \ref{prophpq}.  Combining \reff{5}
and the conclusions in the previous four lemmas, Lemma \ref{lem401}-Lemma
\ref{lem404},  we get
\begin{eqnarray*}
  \begin{split}
   \norm{\p_q^{m-n}\p_p^{n+2} h}_{0}&\leq \inner{(c_1+c_3+c_4)L_1^{-1}+c_2L_1L_2^{-1}}L_1^{\alpha_1-1}L_2^{\alpha_2} [(\abs{\alpha}-2)!]^{s}\\
   &\leq \frac{1}{2}L_1^{\alpha_1-1}L_2^{\alpha_2} [(\abs{\alpha}-2)!]^{s},
   \end{split}
\end{eqnarray*}
where in the last inequality we chose
\begin{equation}\label{Lgeq}
L_1\geq 4(c_1+c_3+c_4) , \quad L_2\geq 4c_2L_1.
\end{equation}
Then
we get the desired estimate \reff{+hqmn},  and thus the validity of
$(F_{m,n})$ and $(F_{m})$.
Summarizing  the relations \reff{l1l21}, \reff{l1l22},
\reff{l1l23}, \reff{l1l24}, \reff{l1l25} and \reff{Lgeq}, we can choose
\begin{eqnarray}\label{L1 L2}
  L_1\geq \max\Big\{L, \norm{h}_4+1,  4(c_1+c_3+c_4)\Big\} ~~ {\rm and}~~ L_2\geq \inner{8+4c_2+\tilde M+\tilde N}L_1,
\end{eqnarray}
with $\tilde M, \tilde N$ the  constants appearing respectively in \reff{l1l24} and \reff{l1l25},
to complete the proof of Proposition \ref{prophpq}.
\end{proof}

The rest of this section is devoted to

\begin{proof}[Proof of Lemma \ref{+stab}]
We refer to \cite[Lemma 4.4]{LW} for the proof of the conclusions \reff{conc1}-\reff{conc4}. Now we only prove the final statement \reff{conc5}.
 To  simplify
the notations, we use $\norm{\cdot}$  in this proof
to stand for the  norm $\norm{\cdot}_{C^0(\bar R)}$, and  use $a_j,j\geq 1,$ to denote different
suitable harmless constants larger than $1$,   which depend on the $C^{3,0}$-norm
of $u$ and $C^{2,0}$-norm of $v$,  but are independent of  the order $\alpha_0$ of
derivative.

Assume $u\in\mathcal A_3$ and $v\in\mathcal A_2$.  By Leibniz formula we have, for any  $\alpha\leq \alpha_0$ with
$\abs\alpha\geq 3$,
\begin{eqnarray*}
  \p^\alpha\inner{uv}=\sum_{0\leq \beta\leq\alpha} {\alpha\choose \beta}
   \inner{ \p^{\beta}u}   \inner{ \p^{\alpha-\beta}v}.
\end{eqnarray*}
Then
\begin{eqnarray*}
  \norm{\p^\alpha\inner{uv}} \leq  \sum_{0\leq\beta\leq\alpha} \frac{\abs\alpha!}{\abs\beta !\abs{\alpha-\beta}!}
   \norm{ \p^{\beta}u} \norm{
       \p^{\alpha-\beta}v}
= I_1+I_2+I_3,
\end{eqnarray*}
with
\begin{eqnarray*}
       I_1&=&  \sum_{\stackrel{0\leq\beta\leq\alpha}{\abs\beta\leq 3}}  \frac{\abs\alpha!}{\abs\beta !\abs{\alpha-\beta}!}
   \norm{\p^{\beta}u} \norm{
       \p^{\alpha-\beta}v},\\
   I_2&=& \sum_{\stackrel{0\leq\beta\leq\alpha}{4\leq \abs\beta\leq \abs\alpha-3}}  \frac{\abs\alpha!}{\abs\beta !\abs{\alpha-\beta}!}
   \norm{ \p^{\beta}u} \norm{
       \p^{\alpha-\beta}v},\\
I_3&=&  \sum_{\stackrel{0\leq\beta\leq\alpha}{\abs\beta\geq\abs\alpha-2}}  \frac{\abs\alpha!}{\abs\beta !\abs{\alpha-\beta}!}
   \norm{ \p^{\beta}u} \norm{
       \p^{\alpha-\beta}v}.
\end{eqnarray*}
Since $H_2\geq H_1$, direct computation shows that there exists $a_1>1$ such that
   \begin{eqnarray*}
    I_1+I_3 
    \leq a_1 \inner{\norm{u}_3+\norm{v}_2+1}^2 H_1^{\alpha_1-2}H_2^{\alpha_2}[(\abs\alpha-3)!]^s.
   \end{eqnarray*}
For $I_2$, which appears only when $\abs{\alpha}\geq 7$, we have
\begin{eqnarray*}
    I_2&\leq&  \sum_{\stackrel{0\leq\beta\leq\alpha}{4\leq \abs\beta\leq \abs\alpha-3}}\frac{\abs\alpha!}{\abs\beta !\abs{\alpha-\beta}!}
     H_1^{\beta_1-3}H_2^{\beta_2}\inner{(\abs\beta-4)!}^sH_1^{\alpha_1-\beta_1-2}H_2^{\alpha_2-\beta_2}\com{(\abs{\alpha}- \abs{\beta}-3)!}^s\\
  &\leq& a_2 \sum_{\stackrel{0\leq\beta\leq\alpha}{4\leq \abs\beta\leq \abs\alpha-3}}\frac{\abs\alpha!}{\abs\beta^{4} \abs{\alpha-\beta}^{3}}
     H_1^{\alpha_1-5}H_2^{\alpha_2}\com{(\abs\beta-4)!}^{s-1}\com{(\abs{\alpha}- \abs{\beta}-3)!}^{s-1}\\
  &\leq& a_2 H_1^{\alpha_1-5}H_2^{\alpha_2} \sum_{\stackrel{0\leq\beta\leq\alpha}{4\leq \abs\beta\leq \abs\alpha-3}} \frac{\abs\alpha!}{\abs\beta^{4} \abs{\alpha-\beta}^{3}}
     \com{(\abs{\alpha}-7)!}^{s-1}\\
  &\leq& a_2H_1^{\alpha_1-2}H_2^{\alpha_2}\com{(\abs\alpha-3)!}^{s} \sum_{\stackrel{0\leq\beta\leq\alpha}{4\leq \abs\beta\leq \abs\alpha-3}} \frac{\abs\alpha^{3}}{ \abs\beta^{4}\inner{\abs\alpha-\abs\beta}^{3}}\\
  &\leq& a_3 H_1^{\alpha_1-2}H_2^{\alpha_2}\com{(\abs\alpha-3)!}^{s},
\end{eqnarray*}
the last inequality following from Lemma \ref{lemSum}. In view of the
estimates for $I_1,I_2$ and $I_3$, we have,  for any  $\alpha\leq
\alpha_0$  with $\abs\alpha\geq 3$,
\begin{equation}\label{uvnorm}
  \norm{\p^{\alpha}(uv)}\leq a_4\inner{\norm{u}_3+\norm{v}_2+1}^2 H_1^{\alpha_1-2}H_2^{\alpha_2} \com{(\abs\alpha-3)!}^{s}
\end{equation}
with  $a_4= a_1+ a_3$.
Thus the conclusion \reff{conc5} follows if
we choose $\tilde c_*\geq a_4\inner{\norm{u}_3+\norm{v}_2+1}^2$.
\end{proof}

\section{Regularity of stratified water waves without surface tension}\label{sec5}
Without surface tension, that is, when $\sigma=0$, the flow is driven only by the gravity. Then the  equation \reff{EquPsi2} becomes
\begin{equation}\label{EquPsi22}
  \abs{\nabla
    \psi}^2+2g(y+d)\rho=Q,\qquad
  y=\eta(x).  \tag{3b$'$}
  \end{equation}
Correspondingly, the equation \reff{Equh2} becomes
\begin{equation}\label{Equh22}
  1+h_q^2+(2g\rho h-Q)h_p^2=0,\quad
  {\rm on~~} p=0. \tag{4b$'$}
\end{equation}

Proven in this section is the regularity property of all
the streamlines and the pseudo-stream function of stratified water waves without surface
tension.

\begin{theorem}\label{th3}
   Consider the free boundary problem \reff{EquPsi1}-\reff{EquPsi4} with
   \reff{EquPsi2} replaced by  the above \reff{EquPsi22}.  Suppose $\beta\in C^{0,\mu}([p_0,0])$ and $\rho\in
   C^{1,\mu}([p_0,0])$ with $\mu$ and $p_0$ given. Then
each streamline including the free surface $y=\eta(x)$ is a real-analytic curve.   If, in addition,
$\beta,\rho \in G^s([p_0, 0])$ with $s\geq 1$,   then $\psi(x,y) \in
G^s(\bar \Omega )$; in particular if $s=1$, i.e., $\beta$ and $\rho$ are analytic in $[p_0,
0]$, then the pseudo-stream function $\psi(x,y)$ is analytic in $\bar\Omega$.
\end{theorem}

\begin{proof}

As before we only prove the corresponding regularity for height
function $h$ of the system \reff{Equh1}-\reff{Equh3} with \reff{Equh2}
replaced by the above \reff{Equh22}.   Since the arguments are nearly the same as those in presence of surface tension (Section
\ref{sec3} and Section \ref{sec4}),  we shall only give a sketch and indicate how to modify the analysis.

Repeating the arguments in Section \ref{sec4},  we can derive  the
second statement in Theorem \ref{th3},  without any difference. So we
only need to focus on the proof of the first statement on the analyticity of streamlines, which, as we will see below, is in fact simpler than its counterpart  in Section \ref{sec3}.

As in
Remark \ref{remReg} we may assume $\p_q^k h\in C^{2,\mu}(\bar R)$ for
any $k\in\N$.   Taking $m^{th}$-order derivative with respect to the $q$-variable on both sides of the equation \reff{Equh22} shows that the second equation in \reff{oprtAB} becomes
\begin{eqnarray*}
  \tilde B(h)[\p_q^m h]=\tilde\varphi_1+\tilde\varphi_2,\quad {\rm on}~ p=0,
\end{eqnarray*}
with the operator
\begin{eqnarray*}
  \tilde B(h)[\phi]=h_q\phi_q+(2g\rho h-Q)h_p\phi_p+2gh_p^2\phi
\end{eqnarray*}
and the right-hand side
\begin{eqnarray*}
  \tilde\varphi_1=-\frac{1}{2}\sum_{1\leq n\leq m-1}{m\choose n}(\p_q^n h_q)(\p_q^{m-n} h_q)-\frac{1}{2}(2g\rho h-Q)\sum_{1\leq n\leq m-1}{m\choose n}(\p_q^n h_p)(\p_q^{m-n} h_p),
  \end{eqnarray*}
  \begin{eqnarray*}
   \tilde\varphi_2=-\rho g\sum_{1\leq n\leq m-1}{m\choose n}(\p_q^n h)(\p_q^{m-n} h^2_p).
\end{eqnarray*}
The first and third equations in \reff{oprtAB} remain unchanged. Then,
as before, our aim is to show that the corresponding estimate as  in Proposition \ref{propassu} holds, that is, there exists a  constant $\tilde L\geq 1$ such
that for any $m\geq2,$
\begin{eqnarray*}
(\tilde E_m): \quad\quad \norm{\p^m_q h}_{2,\mu}\leq
\tilde L^{m-1}(m-2)!,
\end{eqnarray*}
and to this end the main point is to show that $(\tilde E_m)$ holds under the assumption that for
any $j$ with $2\leq j\leq m-1$, the following estimates
\begin{eqnarray}\label{tildIne}
  (\tilde E_j):\quad\quad \norm{\p^j_q h}_{2,\mu}\leq
  \tilde L^{j-1}(j-2)!
\end{eqnarray}
are already valid.

Since $h\in C^{2,\mu}(\bar R)$, the coefficients
of the operator $\tilde B(h)$
are in $C^{1,\mu}(\bar R)$. Moreover by the induction assumption
\reff{tildIne},
$\tilde\varphi_1$ and $\tilde\varphi_2$ are in $C^{1,\mu}(\bar R)$. Furthermore in this case the operator $\tilde B(h)$ satisfies the complementing condition in the sense of \cite{MR0125307} since the coefficient $(2gh-Q)h_p$ of $\phi_p$ is nonzero and satisfies
\begin{eqnarray*}
  (2g\rho h-Q)h_p=\frac{1+h_q^2}{h_p}\geq \frac{1}{\sup_{\bar R}h_p}\geq \delta
\end{eqnarray*}
in view of the boundary condition \reff{Equh22} and \reff{hp
  bounded}. As a result,
we can apply the Schauder estimate (see for instance \cite[Theorem 6.30]{MR1814364} and \cite{MR0125307})   to conclude
\begin{eqnarray}\label{shauder2}
  \norm{\p_q^m h}_{2,\mu}\leq \tilde {\mathcal C}\inner{\norm{\p_q^m h}_0+\sum_{i=1}^2\norm{f_i}_{0,\mu}+\sum_{i=1}^2\norm{\tilde\varphi_i}_{1,\mu}},
\end{eqnarray}
with $\tilde {\mathcal C}$ a constant independent of $m$,   and $f_i$, $i=1,2$, defined in
\reff{def f1}-\reff{def f2}.   Then we can use  the similar
arguments as in Section \ref{sec3} without any additional
difficulty,   to  conclude
\[
\tilde {\mathcal C}\inner{\norm{\p_q^m
  h}_0+\sum_{i=1}^2\norm{f_i}_{0,\mu}+\sum_{i=1}^2\norm{\tilde\varphi_i}_{1,\mu}}\leq
\tilde C_1 \tilde L^{m-2}(m-2)!,
\]
with  $\tilde C_1$ a constant independent of  $m$. Choosing $\tilde L\geq \tilde C_1$, we complete the proof of the proposition.
\end{proof}

\section*{Acknowledgements} The support of Foundation of WUST and support of Hubei Province Key Laboratory of SSMP are gratefully acknowledged.
\bibliographystyle{plain}

\begin{thebibliography}{99}

\bibitem{MR0125307}
S.~Agmon, A.~Douglis, and L.~Nirenberg.
\newblock \emph{Estimates near the boundary for solutions of elliptic partial
  differential equations satisfying general boundary conditions. {I}},
\newblock { Comm. Pure Appl. Math.},  {12} (1959) 623--727.

\bibitem{clx2011}
H. Chen, W.-X. Li  and C.-J. Xu.
\newblock \emph{Gevrey regularity of subelliptic Monge-Amp\`ere equations in the plane}.
\newblock {Advances in  Mathematics },  {228} (2011) 1816--1841.

\bibitem{MR2753609}
Adrian Constantin and Joachim Escher.
\newblock \emph{Analyticity of periodic traveling free surface water waves with
  vorticity},
\newblock { Ann. of Math. (2)},  {173(1)} (2011) 559--568.

\bibitem{MR2027299}
A. Constantin and W. Strauss.
\newblock \emph{Exact steady periodic water waves with vorticity},
\newblock Comm. Pure Appl. Math.,  {57}
(2004), 481--527.

\bibitem{Esurvey}
J. Escher.
\newblock \emph{Regularity of rotational traveling water waves},
\newblock preprint.

\bibitem{EMM}
J. Escher, A.-V. Matioc and B.-V. Matioc.
\newblock \emph{On stratified steady periodic water waves with linear density distribution and stagnation points},
\newblock to appear in J. Differential Equations.

\bibitem{MR1814364}
D. Gilbarg and Neil~S. Trudinger.
\newblock {\em Elliptic partial differential equations of second order}.
\newblock Classics in Mathematics. Springer-Verlag, Berlin, 2001.
\newblock Reprint of the 1998 edition.

\bibitem{HurImrn}
Vera~Mikyoung Hur.
\newblock \emph{Analyticity of Rotational Flows Beneath Solitary Water Waves},
\newblock { Int Math Res Notices 2011.},  doi: 10.1093/imrn/rnr123.

\bibitem{MR2763714}
D. Henry.
\newblock Analyticity of the streamlines for periodic travelling free surface
  capillary-gravity water waves with vorticity,
\newblock {\em SIAM J. Math. Anal.},  {42(6)} (2011) 3103--3111.

\bibitem{HenryJmfm}
D. Henry.
\newblock \emph{Analyticity of the Free Surface for Periodic Travelling Capillary-Gravity Water Waves with Vorticity},
\newblock { Journal of Math. Fluid Mech.},  2011,  DOI: 10.1007/s00021-011-0056-z.

\bibitem{HenryPro}
D. Henry.
\newblock \emph{Regularity for steady periodic capillary water waves with vorticity},
\newblock {Phil. Trans. Roy. Soc. A}, to appear.
  
\bibitem{HM}
D. Henry and B. Matioc.
\newblock \emph{On the existence of steady periodic capillary-gravity stratified water waves},
\newblock {Ann. Scuola Norm. Sup. Pisa}, to appear.

\bibitem{HMdcds}
D. Henry and B. Matioc.
\newblock \emph{On the regularity of steady periodic stratified water waves},
\newblock {Comm. Pure Appl. Anal.}, to appear.

\bibitem{MR531272}
D.~Kinderlehrer, L.~Nirenberg, and J.~Spruck.
\newblock \emph{Regularity in elliptic free boundary problems},
\newblock { J. Analyse Math.},  {34} (1978) 86--119 (1979).


\bibitem{MR0049399}
H. Lewy.
\newblock \emph{A note on harmonic functions and a hydrodynamical application},
\newblock { Proc. Amer. Math. Soc.},  {3} (1952)111--113.

\bibitem{LW}
W.-X. Li and L.-J. Wang.
\newblock \emph{Regularity of traveling free surface water waves with vorticity},
\newblock {http://arxiv.org/abs/1112.2730}.



\bibitem{Long53}
R. R. Long.
\newblock \emph{Some aspects of the flow of stratified fluids. Part I : A theoretical investigation},
\newblock { Tellus}  {5}
(1953), 42--57.

\bibitem{MaImrn}
B.-V. Matioc.
\newblock \emph{Analyticity of the Streamlines for Periodic Traveling Water Waves with Bounded Vorticity},
\newblock { Int Math Res Notices},  {17}
(2011), 3858-3871. doi: 10.1093/imrn/rnq235

\bibitem{MaQam}
B.-V. Matioc.
\newblock On the regularity of deep-water waves with general vorticity
distributions,
\newblock to appear in {\em Quart.  Appl. Math.}, 2011.

\bibitem{MR1249275}
Luigi Rodino.
\newblock {\em Linear partial differential operators in {G}evrey spaces}.
\newblock World Scientific Publishing Co. Inc., River Edge, NJ, 1993.

\bibitem{Turner81}
R. E. L. Turner.
\newblock \emph{Internal waves in fluids with rapidly varying density},
\newblock { Ann. Suola Norm. Sup. PisaCl. Si.},  {8}
(1981), 513--573.



\bibitem{Walsh1}
S. Walsh.
\newblock \emph{Stratified steady periodic water waves},
\newblock { SIAM J. Math. Anal.},  {41}
(2009), 1054--1105.

\bibitem{Walsh2}
S. Walsh.
\newblock \emph{Steady periodic gravity waves with surface tension},
\newblock {arXiv:0911.1375, 2009}.


\bibitem{Yih60}
C.-S. Yih.
\newblock \emph{Exact solutions for steady two-dimensional flow of a stratified fluid},
\newblock { J. Fluid Mech.},  {9}
(1960), 161--174.

\end{thebibliography}

\end{document}